\documentclass[10pt]{amsart}

\usepackage{amsmath, amsthm}
\usepackage{eucal}

\usepackage{amssymb}
\usepackage{amscd}
\usepackage{palatino}
\usepackage{latexsym}
\usepackage{epsfig}
\usepackage{graphicx}
\usepackage{amsfonts}
\usepackage{psfrag}

\usepackage{url}
\usepackage{cite}

\usepackage[margin=1in]{geometry}

\input xy
\xyoption{all}
\UseComputerModernTips

\numberwithin{equation}{section}
\numberwithin{figure}{section}

\newtheorem{theorem}{Theorem}[section]
\newtheorem{lemma}[theorem]{Lemma}
\newtheorem{proposition}[theorem]{Proposition}
\newtheorem{corollary}[theorem]{Corollary}

\newtheorem{remark}[theorem]{Remark}
\newtheorem{example}[theorem]{Example}
\newtheorem{question}[theorem]{Question}

\theoremstyle{definition}
\newtheorem{definition}[theorem]{Definition}

\newcommand{\C}{{\mathbb{C}}}

\newcommand{\Z}{{\mathbb{Z}}}

\newcommand{\R}{{\mathbb{R}}}
\newcommand{\N}{{\mathbb{N}}}

\renewcommand{\P}{{\mathbb{P}}}
\newcommand{\F}{{\mathbb{F}}}

\renewcommand{\t}{\mathfrak{t}}
\newcommand{\g}{\mathfrak{g}}

\renewcommand{\b}{\mathfrak{b}}

\newcommand{\into}{\hookrightarrow}

\newcommand\lie[1]{{\mathfrak #1}}

\DeclareMathOperator{\Lie}{Lie}
\DeclareMathOperator{\image}{image}

\DeclareMathOperator{\Sym}{Sym}

\DeclareMathOperator{\pt}{pt}

\DeclareMathOperator{\supp}{supp}

\newcommand{\hsm}{{\hspace{1mm}}}

\newcommand{\Flags}{\mathcal{F}\ell ags}

\begin{document}

\title{Poset pinball, GKM-compatible
  subspaces, 
  and Hessenberg varieties}

\author{Megumi Harada}
\address{Department of Mathematics and
Statistics\\ McMaster University\\ 1280 Main Street West\\ Hamilton, Ontario L8S4K1\\ Canada}
\email{Megumi.Harada@math.mcmaster.ca}
\urladdr{\url{http://www.math.mcmaster.ca/Megumi.Harada/}}
\thanks{The first author is partially supported by an NSERC Discovery Grant,
an NSERC University Faculty Award, and an Ontario Ministry of Research
and Innovation Early Researcher Award.  The second author is partially supported by 
 NSF grant DMS-0801554, a Sloan Research Fellowship, and an Old Gold Research Fellowship.}

 \author{Julianna Tymoczko}
 \address{Department of Mathematics, University of Iowa, 14 MacLean
 Hall, Iowa City, Iowa 52242-1419, U.S.A.}
 \email{tymoczko@uiowa.edu}
\urladdr{\url{http://www.math.uiowa.edu/~tymoczko/}}


\date{\today}


\begin{abstract}

This paper has three main goals. First, we set up a general framework
to address the problem of constructing module bases for the equivariant cohomology of
certain subspaces of GKM spaces. To this end we introduce the notion
of a \textbf{GKM-compatible} subspace of an ambient GKM space. We 
also discuss \textbf{poset-upper-triangularity}, a key combinatorial
notion in both GKM theory and more generally in localization theory in
equivariant cohomology. With
a view toward other applications, we present
parts of our setup in a general algebraic and combinatorial
framework. Second, motivated by our central problem of building module
bases, we introduce a combinatorial game which we dub
\textbf{poset pinball} and illustrate with several examples. Finally,
as first applications, 
we apply the perspective of GKM-compatible subspaces and poset pinball
to construct explicit and computationally convenient module bases for
the $S^1$-equivariant cohomology of all Peterson varieties of
classical Lie type, and subregular Springer varieties of Lie type
$A$. In addition, in the Springer case we use our module basis to lift the classical Springer representation on the
ordinary cohomology of subregular Springer varieties to $S^1$-equivariant cohomology in Lie type $A$.

\end{abstract}

\maketitle

\setcounter{tocdepth}{1}
\tableofcontents

\section{Introduction}\label{sec:intro}

This manuscript has three main goals. First, we develop a general
framework and perspective to construct
computationally convenient module bases for the equivariant cohomology
of certain spaces equipped with group actions. In particular, we
introduce the notion of a \textbf{GKM-compatible subspace} of an
ambient GKM space. While not themselves GKM spaces, GKM-compatible subspaces allow
us to exploit the combinatorial advantages
of GKM theory applied to the ambient GKM space.   We primarily use  Borel-equivariant
cohomology with field coefficients, but expect future applications in
other generalized
equivariant cohomology theories.  For this reason, we present part of this framework
in an abstract algebraic 
setting, formalizing algebraic properties
of the equivariant cohomology of GKM spaces in the language of
submodules of product modules indexed by a graded partially ordered
set. We also discuss the crucial notion of
\textbf{poset-upper-triangular} subsets of a module, and give one
possible answer to a question of Billey's by providing examples of
topological spaces with no combinatorially-natural
poset-upper-triangular basis. 

Second, we
introduce a combinatorial game we call \textbf{poset
  pinball}. The game is designed to address some of the difficulties
which arise in the analysis of the equivariant cohomology of
GKM-compatible subspaces, but the game itself is purely combinatorial
and does not depend on the motivating geometry.

Third, as applications, we use the above theory
to describe some \textbf{nilpotent Hessenberg varieties}, which are
a rich class of algebraic varieties arising in geometric
representation theory.  GKM theory does not directly
apply to nilpotent Hessenberg varieties but our
methods do.  We first
prove in Theorem~\ref{theorem:Hessies are GKM-compatible} that the \textbf{nilpotent
    Hessenberg varieties $\mathcal{H}(N,H)$} in Lie type $A$, and in
  general Lie type given an extra condition on the parameter $N$, are
 GKM-compatible subspaces of the flag
  variety. 
Then in Theorem~\ref{theorem: classical basis for
    Peterson} we use poset pinball to construct explicit
  module bases for the $S^1$-equivariant
cohomology rings of \textbf{Peterson varieties in all classical Lie
  types}; this generalizes earlier work in the Lie type $A$ case
\cite{HarTym09}. 
Similarly, we use poset pinball in Theorem~\ref{theorem: pinball for subregular} to construct module bases for the
  $S^1$-equivariant cohomology of \textbf{subregular Springer varieties in
Lie type $A$} (also studied by Slodowy \cite{Slo80}). We then 
use this `poset pinball basis' and Kostant-Kumar's 
  $S_n$-action on $H^*_T(\Flags(\C^n);\C)$ to explicitly construct in
Corollary~\ref{corollary:equivariant Springer rep} a new geometric
representation of $S_n$ on $H^*_{S^1}(\mathcal{S}_N;\C)$ and also 
prove that it lifts the well-known Springer representation on
the ordinary cohomology of subregular Springer varieties to $S^1$-equivariant cohomology.

Our work develops out of GKM theory, named for the
influential manuscript of Goresky-Kottwitz-MacPherson \cite{GKM}. If
$X$ is a suitable $G$-space, GKM theory gives a combinatorial
description of the generalized equivariant cohomology ring $E^*_G(X)$
via restriction to $E^*_G(X^G)$.  If the $G$-space $X$ has isolated
fixed points, there often exists a computationally convenient module
basis for $E^*_G(X)$; a classical example is the set of
\emph{(equivariant) Schubert classes} $\{\sigma_w\}_{w \in W}$ which
form a basis for $H^*_T(G/B)$ where $G$ is a reductive complex
algebraic group and $B$ is a Borel subgroup. Current research in
equivariant topology (especially Schubert calculus) frequently
exploits combinatorial properties of these module bases to obtain
topological information, e.g. the product structure of equivariant
cohomology rings (see for instance \cite{KnuTao03, Kre05, LakRagSan06,
  Wil06, IkeNar09}). However the so-called ``GKM conditions'' which
guarantee that GKM theory applies to a $G$-space $X$ are
stringent 
(see Section~\ref{subsec:background-GKM}), so the theory is restricted
in scope. 
On the other hand, many topological spaces are subspaces of a $G$-space $X$ for which the GKM package holds, since for instance we may take $X= \mathbb{P}^n$ with the standard torus action.  
Definition~\ref{def:GKM compatible} introduces the notion of
a \textbf{GKM-compatible subspace} $Y$ of the $G$-space $X$, equipped with the action of a subgroup $H \subseteq G$.  We show that we can use GKM theory
on the ambient space $X$ in order to draw conclusions about the
$H$-equivariant topology of $Y$.  For example, when $G=T$ is a torus and $H = S$ is a subtorus, we
present two concrete constructions of combinatorial bases for
$H^*_{S}(Y)$ given a suitable basis for $H^*_T(X)$:
Proposition~\ref{prop:pinball-betti} gives a {\em poset pinball basis}
and Theorem~\ref{theorem:matching} gives a {\em matching basis}.

One of the goals of this manuscript is to formalize some
of the features of GKM theory 
into purely algebraic and combinatorial
terms. We believe that the separation of the algebra and combinatorics
from the specifics of the geometry serves to clarify some of the issues involved. 
We start with a poset $\mathcal{I}$ satisfying
conditions arising naturally in
geometric applications. 
We then place the GKM description of equivariant cohomology rings in
the more general algebraic setting of a submodule $M$ of a product
module $\prod_{i \in \mathcal{I}} M_i$ whose factors are indexed by
the poset $\mathcal{I}$.  We also discuss one of the core notions of
this manuscript, namely \textbf{poset-upper-triangularity}, defined
precisely in Definition~\ref{def:poset-upper-triangular}.  Roughly, a
subset $\{x_{\alpha}\}_{\alpha \in A} \subset \prod_{i \in
  \mathcal{I}} M_i$ is poset-upper-triangular if for each $\alpha$
there exist distinct $i_\alpha \in \mathcal{I}$ such that
\(x_\alpha(j) = 0\) for all \(j \not \geq i_\alpha\) in the poset. In
many contexts, geometric classes in equivariant cohomology give rise to
poset-upper-triangular subsets, like the equivariant Schubert classes
for flag varieties, or more generally cohomology classes obtained from
Morse flows with respect to moment maps on a symplectic manifold. We
present in Theorem~\ref{theorem:poset-upper-triangular conditions} more general circumstances under which
poset-upper-triangular module bases exist for Borel-equivariant
cohomology.  In this algebraic formalism, the analogue of a
GKM-compatible subspace $Y$ of $X$ is a subset $\mathcal{J} \subseteq
\mathcal{I}$ of an ambient poset $\mathcal{I}$
and a homomorphism $\prod_{i \in \mathcal{I}} M_i \to \prod_{j \in
  \mathcal{J}} M'_j$ which is zero on the
factors $i \in \mathcal{I}$ with $i \not \in \mathcal{J}$. 
Our central problem, recorded in an algebraic context in
Question~\ref{question:main-algebraic} and in a geometric context in Question~\ref{question:main}, is that a poset-upper-triangular subset of
$\prod_{i \in \mathcal{I}} M_i$ may not be poset-upper-triangular when
restricted to the components indexed by $\mathcal{J}$.  This is
precisely the issue which our \textbf{poset
  pinball} and its variations are designed to address.
While geometric in inspiration, we emphasize that poset pinball only 
requires the combinatorial data of a poset $(\mathcal{I}, <)$ and a choice
of initial subset $\mathcal{J} \subseteq \mathcal{I}$.

As a consequence of the examples of poset pinball games computed in
this manuscript, we also propose a 
perspective on poset-upper-triangularity which somewhat differs from
that which may be most natural from the point of view of
combinatorics (see Remark~\ref{remark:billey}). 
Combinatorists view poset-upper-triangularity as a key computational
property; indeed, Billey suggests that it is one of the essential features of
the Schubert basis and asks for constructions of such poset-upper-triangular bases in
equivariant cohomology rings of other $G$-spaces. 
On the other hand, in our poset pinball examples it can happen that we
obtain subsets of modules that are \emph{not} poset-upper-triangular with
respect to the original partial order $<$ but are nevertheless 
poset-upper-triangular with respect to a total order $\prec$
compatible with the original partial order.
Poset-upper-triangularity with respect to a total order often suffices to guarantee that a subset
 is linearly independent and hence a module basis,
so in some geometric contexts, it may be more natural to
require only
that module bases be upper-triangular with respect to some
choice of total order compatible with the original partial order.

We now present a concrete 
example of poset pinball in order to convey the flavor of the
game. 
Let $\mathcal{I} = S_4$ denote the permutation group $S_4$. Elements
of $\mathcal{I} = S_4$ are the vertices in
Figure~\ref{fig:SpringerN=4,YoungX=2,2} and are labelled by a choice
of reduced-word decomposition.  (We omit some of the elements of $S_4$
in the figure
because they are not relevant in this example.)  The set $\mathcal{I}$
is partially ordered by Bruhat order, so we draw an edge between
vertices \(w, w' \in \mathcal{I}\) if and only if \(w < w'\) and there
is no \(w'' \in \mathcal{I}\) with \(w < w'' < w'.\) The vertices are
drawn so that the poset's minimal element $e$ is at the bottom, and
horizontal levels indicate Bruhat length.  Let $\mathcal{J}$ be the
subset of $\mathcal{I}$ indicated by the circled vertices in
Figure~\ref{fig:SpringerN=4,YoungX=2,2}, so there are $6$ elements in
the subset $\mathcal{J}$.  To play poset pinball, we successively
release a circled vertex, starting from the lowest vertex in
$\mathcal{J}$ and then moving up; we imagine each circled vertex
rolling down along the edges of the poset until it comes to rest in
a lowest-possible unoccupied vertex, at which point the circle turns into a
square.  (See Section~\ref{subsec:pinball} for precise statements.) This results in a
choice of $6$ vertices of $\mathcal{I}$ corresponding to the original
elements of $\mathcal{J}$ and indicated by the squared vertices in the
figure below. Table~\eqref{eq:SpringerN=4,YoungX=2,2-pinball}  records the exact correspondence
between the initial vertices \(w \in \mathcal{J}\) and the squared
vertices \(v \in \mathcal{I}\).

\begin{figure}[h]
\begin{picture}(150,122)(0,0)
\put(75,0){\circle*{5}}
\multiput(45,30)(30,0){3}{\circle*{5}}
\multiput(15,60)(30,0){5}{\circle*{5}}
\multiput(45,90)(60,0){2}{\circle*{5}}
\put(75,120){\circle*{5}}

\put(75,0){\line(-1,1){30}}
\put(75,0){\line(0,1){30}}
\put(75,0){\line(1,1){30}}

\put(45,30){\line(-1,1){30}}
\put(45,30){\line(0,1){30}}
\put(45,30){\line(1,1){30}}
\put(75,30){\line(-1,1){30}}
\put(75,30){\line(-2,1){60}}
\put(75,30){\line(1,1){30}}
\put(75,30){\line(2,1){60}}
\put(105,30){\line(-1,1){30}}
\put(105,30){\line(0,1){30}}
\put(105,30){\line(1,1){30}}

\put(45,90){\line(0,-1){30}}
\put(45,90){\line(1,-1){30}}
\put(45,90){\line(2,-1){60}}
\put(105,90){\line(1,-1){30}}
\put(105,90){\line(-1,-1){30}}
\put(105,90){\line(-3,-1){90}}
\put(75,120){\line(-1,-1){30}}
\put(75,120){\line(1,-1){30}}

\put(75,0){\circle{10}}
\put(75,30){\circle{10}}
\put(45,60){\circle{10}}
\put(105,60){\circle{10}}
\put(45,90){\circle{10}}
\put(75,120){\circle{10}}

\put(69,-6){\framebox(12,12)}
\multiput(39,24)(30,0){3}{\framebox(12,12)}
\multiput(9,54)(60,0){2}{\framebox(12,12)}

\put(84,-3){$e$}
\put(113,26){$s_3$}
\put(60,20){$s_2$}
\put(28,26){$s_1$}
\put(-11,56){$s_1s_2$}
\put(144,56){$s_3s_2$}
\put(26,72){\small $s_2s_1$}
\put(102,67){$s_2s_3$}
\put(68,48){\small $s_1s_3$}
\put(112,89){$s_1s_3s_2$}
\put(10,89){$s_2s_1s_3$}
\put(30,116){$s_2s_1s_3s_2$}

\end{picture}
\caption{An instance of poset pinball, for the Springer variety in
  $\Flags(\C^4)$ specified by a nilpotent operator $N$ corresponding
  to the partition $(2,2)$.}\label{fig:SpringerN=4,YoungX=2,2}
\end{figure}
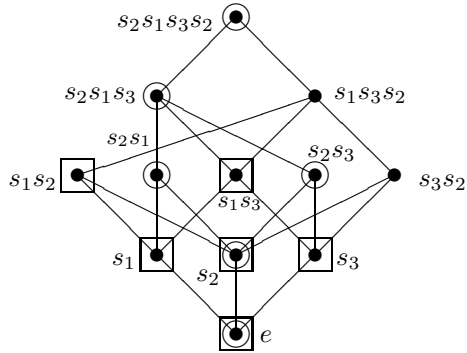

\renewcommand{\arraystretch}{1.3}
\begin{equation}\label{eq:SpringerN=4,YoungX=2,2-pinball}
\begin{array}{c||c|c|}
\mbox{pinball step} & w_k & v_k  \\ \hline \hline
1 & w_1 = e  & v_1 = e \\ \hline
2 & w_2 = s_2 & v_2 = s_2  \\ \hline
3 & w_3 = s_2 s_3 & v_3 = s_3  \\ \hline
4 & w_4 = s_2 s_1 & v_4 = s_1  \\ \hline
5 & w_5 = s_2 s_1 s_3 & v_5 = s_1 s_3 \\ \hline
6 & w_6 = s_2 s_1 s_3 s_2 & v_6 = s_1 s_2 \\ \hline 
\end{array}
\end{equation}

In fact, let $\mathcal{S}_N \subseteq \Flags(\C^4) \cong GL(4,\C)/B$
denote the Springer variety of type $A$ corresponding to a nilpotent
endomorphism $N$ with two Jordan blocks of size $2$. Let 
the maximal torus $T$ act in the standard way on $\Flags(\C^4)$. An $S^1$
subtorus of $T$ preserves $\mathcal{S}_N$ and it turns out that the
$S^1$-fixed points $\mathcal{S}_N^{S^1} = \mathcal{J}$ are exactly the
subset of $\mathcal{I} = S_4 \cong \Flags(\C^4)^T$ indicated in Figure
\ref{fig:SpringerN=4,YoungX=2,2}.  (See Section~\ref{sec:springer}.) 
Moreover in this case the choices
$\{v_i\}_{i=1}^6$ obtained via poset pinball give rise to an
$H^*_{S^1}(\pt;\F)$-module basis for
$H^*_{S^1}(\mathcal{S}_N;\F)$. 
Specifically, consider the ring homomorphism
\begin{equation}\label{eq:intro}
H^*_T(\Flags(\C^4);\F) \to H^*_{S^1}(\mathcal{S}_N;\F)
\end{equation}
obtained by composing the map
\(H^*_T(\Flags(\C^4);\F) \to H^*_{S^1}(\Flags(\C^4);\F)\) with the map
\(H^*_{S^1}(\Flags(\C^4);\F) \to H^*_{S^1}(\mathcal{S}_N;\F)\)
induced by inclusion of groups \(S^1 \into T\) and spaces
\(\mathcal{S}_N \into \Flags(\C^4)\) respectively. Denote the image of
an equivariant Schubert class $\sigma_w \in H^*_T(\Flags(\C^4);\F)$
under the map~\eqref{eq:intro} by $p_w$. Then the $6$ classes $p_{v_1},
\ldots, p_{v_6}$ which correspond to the outcome of the previous poset pinball
game form an
$H^*_{S^1}(\pt;\F)$-module basis for
$H^*_{S^1}(\mathcal{S}_N;\F)$. Moreover, the set of images of the $p_{v_i}$
under the natural restriction map 
\begin{equation}\label{eq:intro-local}
\iota^*: H^*_{S^1}(\mathcal{S}_N;\F) \into H^*_{S^1}(\mathcal{S}_N^{S^1};\F) \cong
\bigoplus_{i=1}^6 H^*_{S^1}(\pt;\F) \cong \bigoplus_{i=1}^6 \F[t]
\end{equation}
is poset-upper-triangular with respect to the partial
order on the fixed points $\mathcal{S}_N^{S^1}$ induced from
Bruhat order. In other words, for each \(p \in
H^*_{S^1}(\mathcal{S}_N;\F)\) let $p(w_i)$ denote the component of
$\iota^*p$ in the $w_i$-th summand of the right side
of Equation~\eqref{eq:intro-local}. Then for each \(i = 1,2,\ldots, 6\) we
have 
\begin{equation}\label{eq:vanishing}
p_{v_i}(w_j) = 0 \mbox{ for } w_j \not \geq w_i 
\end{equation}
where the inequality indicates the partial order on
$\mathcal{S}_N^{S^1} \subseteq S_4$ induced from Bruhat order on
$S_4$. These vanishing properties allow us to do explicit computations
in $H^*_{S^1}(\mathcal{S}_N;\F)$ (see Section~\ref{sec:springer} for
applications).

We now briefly outline the contents of the paper. In
Section~\ref{sec:preliminaries} we present the combinatorial and algebraic
preliminaries for the pinball game. Poset pinball itself
is described in detail in Section~\ref{sec:pinball}. We give two
concrete examples of poset pinball in
Section~\ref{subsec:pinball-examples} and make initial observations
concerning the role played by principal order ideals in pinball theory
in Section~\ref{subsec:ideals}. We then explain the geometric
motivation and context
in Section~\ref{sec:GKM-geometry}.  We begin with a brief review of
relevant GKM theory in Section~\ref{subsec:background-GKM}, and then
in Section~\ref{subsec:subspace-setup} define GKM-compatible subspaces
of GKM spaces. 
With a view toward future
work, we keep the discussion in Sections~\ref{sec:preliminaries}, \ref{subsec:background-GKM},
and~\ref{subsec:subspace-setup} as general as possible. 
Section~\ref{subsec:borel} discusses the case of Borel-equivariant cohomology, which is the main focus of this manuscript. 
The construction of poset
pinball bases for the $S^1$-equivariant cohomology of Peterson
varieties in classical Lie type occupies
Section~\ref{section:petersons}. In Section~\ref{sec:springer}, we
construct a pinball basis for the equivariant cohomology
$H^*_{S^1}(\mathcal{S}_N;\C)$ of the subregular Springer variety of
type $A$, and lift the usual Springer action on
$H^*(\mathcal{S}_N;\C)$ to $H^*_{S^1}(\mathcal{S}_N;\C)$.

Open questions and avenues for future work are mentioned throughout
the manuscript.

\medskip
\noindent \textbf{Acknowledgements.}  We thank 
David Anderson, Darius Bayegan, Barry Dewitt, Rebecca Goldin, and Bridget
Tenner for helpful conversations.  Both authors were supported in part by the
NSF-funded Midwest Topology Network's travel research grant. Moreover,
some of this work was conducted while the first author was a Research
Member at the Mathematical Sciences Research Institute as part of the
Symplectic and Contact Geometry and Topology program in spring
2010. Both authors also benefited from the American Institute of
Mathematics workshop ``Localization techniques in equivariant
cohomology'' held in March 2010. We gratefully acknowledge the support
and generosity of the Midwest Topology Network, MSRI, and AIM.

\section{Combinatorial and algebraic preliminaries}\label{sec:preliminaries}

The main goal of this section is to present, in purely algebraic and
combinatorial terms, the problem that is the primary focus of this
manuscript.  The setup is as follows.  Let $R$ be a commutative ring
and let $(\mathcal{I}, <)$ be a partially ordered set.  Let $M$ be an
$R$-module that can be realized as a submodule of a product $\prod_{i
  \in \mathcal{I}} M_i$ of $R$-modules $M_i$ for each $i \in
\mathcal{I}$.  Our motivating geometric examples arise from
equivariant topology, where the ring \(R = E^*_G(\pt)\) is the
equivariant cohomology ring of a point, the module \(M = E^*_G(X)\) is
that of a $G$-space $X$, and $M$ injects as a submodule into the
product \(\prod_{i \in {\mathcal I}} M_i = \prod_{i \in {\mathcal I}}
E^*_G(\pt)\).  (See Section~\ref{subsec:background-GKM} for geometric
details.) However, in order to highlight the algebraic issues, we have
chosen to keep the geometry out of the discussion until
Section~\ref{sec:GKM-geometry}.

With this in mind, our main problem can be stated as follows: 
given two
modules $M \subseteq \prod_{i \in \mathcal{I}} M_i$ and $M'
\subseteq \prod_{j \in \mathcal{J}} M'_j$ for \(\mathcal{J} \subseteq \mathcal{I}\), together with 
 a homomorphism
$M \longrightarrow M'$ and
a {\em poset-upper-triangular basis} as defined below,
we wish to 
construct computationally convenient
module bases for $M'$. The next two sections provide the necessary background for this
discussion. Section~\ref{subsec:combinatorics} establishes notation and
terminology for posets and product modules indexed by posets.
In particular, we give a precise definition of a
\emph{poset-upper-triangular} subset of such a module.  We make
general observations related to the construction of poset-upper-triangular
module bases in Section~\ref{subsec:algebra}.  
Section~\ref{subsection:images of modules} describes our main problem precisely.

\subsection{Combinatorial preliminaries}\label{subsec:combinatorics}

Let $({\mathcal I}, <)$ be a partially ordered set. 
Recall that for \(i, j \in \mathcal{I}\), we say that $i$
\textbf{covers} $j$ if $j<i$ and, in addition, there is no $i' \in
\mathcal{I}$ with $j<i'<i$.
A \textbf{rank function} $\rho: {\mathcal I}
\rightarrow \mathbb{N}$ is an $\N$-valued function on the poset such
that 
if $i$ covers $j$ then \(\rho(i) = \rho(j)+1.\)  In the case where $\mathcal{I}$ is infinite, we also 
require $\rho(i)>\rho(j)$.  
For \(i \in
\mathcal{I}\), we call $\rho(i)$ the \textbf{rank} of $i$.   A poset $(\mathcal{I},
<, \rho)$ equipped with a
rank function is called a \textbf{graded poset}.

The partially ordered sets $\mathcal{I}$ in this manuscript will always be assumed
to satisfy the following conditions: 
\begin{itemize} 
\item ${\mathcal I}$ is countable, 
  \item ${\mathcal I}$ is graded, and
   \item for any $d \in {\mathbb N}$ the set $\{ i \in {\mathcal I}: \rho(i) \leq d\}$ is finite.
\end{itemize}

We also recall the following \cite[Chapter 3]{Sta97}. 

\begin{definition}
  Given a poset $(\mathcal{I}, <)$ and an element \(i \in \mathcal{I}\), the \textbf{principal order
    ideal $\mathcal{L}_{\mathcal{I}}(i)$ of $i$} is the subset of elements $i' \in \mathcal{I}$ 
 less than or equal to $i$ with respect to $<$.  In other words 
\[
\mathcal{L}_{\mathcal{I}}(i) := \{ i' \in \mathcal{I} \hsm \vert \hsm i' \leq i \}.
\]
Similarly, the \textbf{principal order filter $\mathcal{U}_{\mathcal{I}}(i)$ of $i$}
is the subset of elements $i' \in \mathcal{I}$ greater than or equal
to $i$ with respect to $<$. In other words  
\[
\mathcal{U}_{\mathcal{I}}(i) := \{i' \in \mathcal{I} \hsm \vert \hsm i \leq i' \}.
\]
\end{definition}

The posets appearing in this manuscript arise as indexing sets for
products of modules, so we introduce some terminology for this
situation. 
Let $R$ be a commutative ring. Suppose $\mathcal{I}$ is a poset as
above and $M_i$ an $R$-module for each \(i \in \mathcal{I}.\) 
Suppose $M$ is a submodule of the product module $\prod_{i \in
  \mathcal{I}} M_i$. 
For \(x \in M \subseteq \prod_{i \in {\mathcal I}} M_i\)
we denote by \(x(i) \in M_i\) the component of $x$ in the $i$-th
factor of the direct product. 
For \(x \in M\), let 
\[   \supp(x) := \{i \in {\mathcal  I}: x(i) \neq 0 \in M_i\}   \]
denote the \textbf{support of $x$}, i.e. the
components $i \in \mathcal{I}$ on which $x$ does not vanish.

\begin{definition}\label{def:poset-flowup}
Suppose $M \subseteq \prod_{i \in \mathcal{I}} M_i$ is a module as above.
An element \(x \in M\) is a \textbf{poset-flow-up (with respect to $<$)} if the support of
  $x$ contains a minimal element $i$, the principal order filter of
  which 
  contains $\supp(x)$; in other words $i \in \supp(x) \subseteq
  \mathcal{U}_{\mathcal{I}}(i)$. We denote the (unique) element 
$i$ by $min(x)$ and call it the \textbf{minimum nonzero coordinate of          
(the poset-flow-up) $x$}. 
 \end{definition}

Note that if $x   \in M$ is a poset-flow-up, then
\begin{equation}\label{eq:poset-flowup-condition}
x(j) = 0 \textup{ for all } j \not \geq min(x). 
\end{equation}
Throughout the manuscript, we consider collections of elements in a
module \(M \subseteq \prod_{i \in \mathcal{I}} M_i\) with vanishing
properties similar to those in~\eqref{eq:poset-flowup-condition}. We have the
following.

\begin{definition}\label{def:poset-upper-triangular}
Suppose $M \subseteq \prod_{i \in \mathcal{I}} M_i$ is a module as above. Suppose \({\mathcal B} = \{x_\alpha\} \subseteq M.\) 
The set ${\mathcal B}$ is
  \textbf{poset-upper-triangular} if 
\begin{itemize} 
\item each
  $x_\alpha$ is a poset-flow-up, and 
\item for \(\alpha \neq \beta\), we have \(min(x_\alpha) \neq
  min(x_\beta).\) 
\end{itemize} 
\end{definition}

\subsection{Upper-triangularity for module generators and
  bases}\label{subsec:algebra}

In this section, we construct 
poset-upper-triangular module generators and bases of $R$-modules \(M
\subseteq \prod_{i \in \mathcal{I}} M_i\) from a purely algebraic
viewpoint. An $R$-module basis must
both generate the module and be $R$-linearly independent; we address
the two conditions separately.  We keep the assumptions of
Section~\ref{subsec:combinatorics}.  Throughout, we think of $M$ as a
topological $R$-module, considered as the (inverse) limit of the
submodules \(M \cap \prod_{j\leq i} M_j\).  In particular, the terms
`generator' and `basis' are understood in the topological sense. (If
the poset $\mathcal{I}$ is finite, then this agrees with the usual
notions.)

Suppose $\prec$ is a total ordering on $\mathcal{I}$ 
compatible with the given partial order $<$ on $\mathcal{I}$.
We begin by inductively constructing a set of
generators of \(M \subseteq \prod_{i \in \mathcal{I}} M_i\) that consists
of poset-flow-ups 
with respect to $\prec$. Poset-upper-triangularity with
respect to the total order $\prec$ is a weaker condition than that
with respect to the original partial order $<$ but we will see later
in the manuscript that
upper-triangularity with respect to $\prec$ suffices for many
computational purposes.

\begin{proposition}\label{proposition:generate}
  Let $({\mathcal I},<)$ be a countable graded partially ordered set
  with a finite number of elements of each rank.  Let $R$ be a
  commutative ring, $M_i$ an $R$-module for each \(i \in
  \mathcal{I}\), and $M$ an $R$-submodule \(M \subseteq \prod_{i \in
    {\mathcal I}} M_i.\) Suppose $\prec$ is a total ordering compatible with
   the partial order $<$ on ${\mathcal I}$.
For each \(i \in {\mathcal I}\), define 
\begin{equation}\label{eq:ideal}
{\mathcal V}^i := \{ x \in M \hsm \vert \hsm x(j) = 0 \textup{ for all }j \in {\mathcal I} \textup{ with } j
\prec i  \}
\end{equation}
and denote by $\mathcal{V}^i(i)$ the image of
    ${\mathcal V}^i$ in $M_i$ under the natural projection $M
    \to M_i$. 
Then for each \(i \in {\mathcal I}\) there exist sets ${\mathcal K}_i$
and nonzero elements \(\{x_{i,
    k}\}_{k \in {\mathcal K}_i} \subseteq M\) such that 
  \begin{enumerate}
  \item[(i)] \(x_{i,k}(j) = 0\) for all \( k \in {\mathcal K}_i\)  and $j \in \mathcal{I}$ with \(j \prec i\),
  \item[(ii)] the set \(\{x_{i,k}(i) \}_{k
      \in {\mathcal K}_i} \subseteq M_i\) generates $\mathcal{V}^i(i)$ as an $R$-module.
\end{enumerate} 
Moreover, the union \({\mathcal B} := \bigcup_{i \in {\mathcal I}}
    \{x_{i,k}\}_{k \in {\mathcal K}_i}\) is a set of 
    $R$-module generators of $M$. 
\end{proposition}

\begin{proof}
  Note first that ${\mathcal V}^i$ and $\mathcal{V}^i(i)$ are
  $R$-submodules of $M$ and $M_i$ respectively.  For each \(i \in
  {\mathcal I}\) choose a set of nonzero generators
  $\{y_{i,k}\}_{k \in \mathcal{K}_i}$ of ${\mathcal
    V}^i(i)$ indexed by a set $\mathcal{K}_i$.  For each \(i \in
  {\mathcal I}\) and \(k \in {\mathcal K}_i\) choose an element
  \(x_{i,k} \in {\mathcal V}^i\) that projects to
  \(y_{i,k}\) under the natural map, namely  \(x_{i,k}(i) = y_{i,k}.\)
  By construction, the sets $\mathcal{K}_i$ and elements $x_{i,k}$ for
  \(i \in \mathcal{I}, k \in \mathcal{K}_i\) satisfy conditions (i)
  and (ii) of the proposition.

  It remains to show that the union $\mathcal{B} := \bigcup_{i \in
    \mathcal{I}} \{x_{i,k}\}_{k \in \mathcal{K}_i}$ forms a set of
  $R$-module generators for $M$.  We first claim that for
  any \(i \in {\mathcal I}\), the module $M$ is generated as an
  $R$-module by the elements \({\mathcal B}_i := \displaystyle \bigcup_{j \prec i} \bigcup_{k \in {\mathcal K}_j} \{x_{j,k}\}\),
together with the submodule ${\mathcal V}^i$. 
We proceed by induction.  For the base case, let \(i_0 \in {\mathcal I}\) be the (unique)
minimal element in ${\mathcal I}$ with respect to the total order 
$\prec$.  In this case $\{j \in {\mathcal I}: j \prec i\}$ is empty,
so ${\mathcal V}^{i_0} = M$ and the claim holds.  Now let \(i
\in {\mathcal I}\) and suppose by induction that the claim holds for
all \(j \prec i.\) By our assumptions on $\mathcal{I}$
in Section~\ref{subsec:combinatorics}, the set $\{j \in \mathcal{I}:
j \prec i\}$ is
finite for any total ordering $\prec$ compatible with the partial
order $<$. Hence there exists \(i' \in {\mathcal I}\) which is 
maximal with respect to $\prec$ in \(\{j \in {\mathcal I}: j
\prec i\}\). Now let \(x \in M.\) By the
inductive hypothesis there exist  \(x'' \in {\mathcal V}^{i'}\) and \(a_{j,k} \in R\) for \(j
\prec i', k \in {\mathcal K}_j\)  such that
\begin{equation}\label{eq:x-linearcomb}
x = x'' + \sum_{j \prec i'} \sum_{k \in {\mathcal K}_j} a_{j,k} x_{j,k}.
\end{equation}
Since $\{x_{i',k}(i')\}_{k \in {\mathcal K}_{i'}}$ generate ${\mathcal V}^{i'}(i')$ there exist 
\(a_{i',k} \in R\) for \(k \in {\mathcal K}_{i'}\)  such that 
\begin{equation}\label{eq:def-asubiprimek}
x''(i') = \sum_{k \in {\mathcal K}_{i'}} a_{i',k} y_{i',k}.
\end{equation}
Now define 
\begin{equation}\label{eq:def-xprime}
x' := x'' - \sum_{k \in {\mathcal K}_{i'}} a_{i',k} x_{i', k}.
\end{equation}
Then  by construction
\begin{equation}\label{eq:sol}
x = x' + \sum_{j \preceq i'} \sum_{k \in {\mathcal K}_j} a_{j,k} x_{j,k}
= x'' + \sum_{j \prec i} \sum_{k \in {\mathcal K}_j} a_{j,k} x_{j,k}.
\end{equation}
To prove the claim, we show that $x' \in \mathcal{V}^i$.  Suppose \(j \prec i.\) Then either \(j \prec i'\)
or \(j=i'\).  First suppose \(j \prec i'.\) Since \(x'' \in {\mathcal
  V}^{i'}\) we have \(x''(j) = 0.\) Similarly, each $x_{i', k} \in
{\mathcal V}^{i'}$ so \(x_{i',k}(j) = 0\).  Projection to $M_j$ is
$R$-linear, so 
\[
x'(j) = x''(j) - \sum_{k \in {\mathcal K}_{i'}} a_{i',k} x_{i',k}(j) =
0
\]
as desired.  Now suppose \(j = i'.\) Then by definition of $x_{i',
  k}$ and by Equation~\eqref{eq:def-asubiprimek} we see
\begin{equation}
x'(i')  = x''(i') - \sum_{k \in {\mathcal K}_{i'}} a_{i', k} x_{i',
  k}(i')  = x''(i') - \sum_{k \in {\mathcal K}_{i'}} a_{i', k}
 y_{i', k} = 0.
\end{equation}
Together these mean $x' \in \mathcal{V}^i$. Hence $M$ is generated by the
elements in $\mathcal{B}_i$ together with $\mathcal{V}^i$. 
The result follows from 
\(M = \underleftarrow{\lim} \hsm M \cap
\left(\prod_{j\leq i} M_j \right).\)
\end{proof}

\begin{remark}\label{remark:triangular}
\begin{enumerate} 
\item It is sometimes impossible to find a set of generators of a
  module $M$ which is poset-upper-triangular with respect to
  the original partial order $<$. For example, suppose $M$ is the
  image of the standard diagonal embedding $R \into \prod_{i \in
    \mathcal{I}} R$ with $1 < |\mathcal{I}| < \infty$, and take the
  trivial partial order on $\mathcal{I}$ in which all pairs of
  elements are incomparable and the trivial rank function $\rho \equiv
  0$. 
  Hence the choice of a total order $\prec$ is a crucial step in Proposition~\ref{proposition:generate}.
\item On the other hand, 
  Proposition~\ref{proposition:generate} extends straightforwardly to partially ordered sets
  {\em without} a choice of total order if we assume that the
  vanishing submodules ${\mathcal V}^i$ together generate all of $M$.
\end{enumerate}
\end{remark}

The previous proposition constructed
module generators. 
Our next task is to deal with $R$-linear independence. For the
remainder of the manuscript, we assume that 
\begin{itemize} 
\item $R$ is a domain, and 
\item for all \(i \in \mathcal{I}\) the module $M_i$ is $R$-torsion-free.
\end{itemize}
These assumptions imply that the submodule $M \subseteq
\prod_{i \in \mathcal{I}} M_i$ is also $R$-torsion-free.  (In fact, in
our geometric applications, it is usually the case that $M_i \cong R$
for all $i$ and that $M$ is a free $R$-module.)  We begin by 
giving conditions under which the
construction in Proposition~\ref{proposition:generate} in fact yields
an $R$-module basis.  

\begin{proposition}\label{proposition:independent} 
  Let \(R, {\mathcal I}, M_i, M\) satisfy the conditions of
  Proposition~\ref{proposition:generate}. Assume that $R$ is a domain
  and that the $R$-module $M_i$ is $R$-torsion-free for all \(i \in
  \mathcal{I}\). 
\begin{enumerate}
\item If $\{x_\alpha\}$
  is poset-upper-triangular, then $\{x_\alpha\}$ is $R$-linearly
  independent.
\item Suppose that for all \(i \in \mathcal{I}\), the sets
  $\{x_{i,k}\}_{k \in \mathcal{K}_i}$ constructed in
  Proposition~\ref{proposition:generate} may be chosen to be
  $R$-linearly independent. Then the union $\mathcal{B} = \bigcup_{i
    \in \mathcal{I}} \{x_{i,k}\}_{k \in \mathcal{K}_i}$ is an
  $R$-module basis of $M$. 
\item Suppose that for all \(i \in \mathcal{I}\), the
  index sets $\mathcal{K}_i$ constructed in
  Proposition~\ref{proposition:generate} may be chosen such that
  either $|\mathcal{K}_i| = 0$ or $|\mathcal{K}_i| = 1$. Then the union $\mathcal{B} = \bigcup_{i
    \in \mathcal{I}} \{x_{i,k}\}_{k \in \mathcal{K}_i}$ is an
  $R$-module basis of $M$. 
\item Suppose that $R$ is a principal ideal domain and each $M_i$ is a
  free $R$-module of rank $1$. Then the
  index sets $\mathcal{K}_i$ constructed in
  Proposition~\ref{proposition:generate} may be chosen such that
  either $|\mathcal{K}_i| = 0$ or $|\mathcal{K}_i| = 1$ and the union $\mathcal{B} = \bigcup_{i
    \in \mathcal{I}} \{x_{i,k}\}_{k \in \mathcal{K}_i}$ is an
  $R$-module basis of $M$. 
\end{enumerate} 
\end{proposition}

\begin{proof}
  If $\{x_{\alpha}\}$ is poset-upper-triangular then for each
  $i \in \mathcal{I}$ there is at most one $\alpha$ with
  $min(x_{\alpha})= i$.  Using this, Part (1) follows by 
  induction.  Part (2) is by definition. Part (3) follows
  from Part (2). Part (4) is a special case of Part (3) since, by
  definition, an ideal in 
  a PID is generated by a single element. 
\end{proof}

In our geometric applications, we are frequently in the situation of
Part (4) of Proposition~\ref{proposition:independent}.  More
generally, whenever Part (3) of
Proposition~\ref{proposition:independent} holds, the module basis
elements correspond to elements of the poset, so we may think of the
bases $\mathcal{B}$ as being indexed by the poset ${\mathcal I}$ (or
possibly a subset of $\mathcal{I}$). In this way, the combinatorics of
the poset ${\mathcal I}$ interacts directly with the algebra of the
module $M$ via the basis ${\mathcal B}$. This links the combinatorial
strategies of Section~\ref{sec:pinball} to the algebraic problem of
constructing module bases.  We close the section with a partial
converse to Proposition~\ref{proposition:generate}.

\begin{lemma}\label{lemma: generate}
  Let \(R, {\mathcal I}, M_i, M\) satisfy the conditions of
  Proposition~\ref{proposition:generate}.  Assume that $R$ is a domain
  and $M_i \cong R$ for each $i \in \mathcal{I}$. Suppose that $\{x_i\}_{i
    \in {\mathcal I}}$ is poset-upper-triangular with respect to the partial order on $\mathcal{I}$, that
  \(\min(x_i) = i\) for each $i$, and that $\{x_i\}_{i \in \mathcal{I}}$ generates $M$ as an
  $R$-module.  Then the choice of singleton set $\{x_i\}$ for each
  $i\in \mathcal{I}$ satisfies Conditions (i) and (ii) of Proposition
  \ref{proposition:generate}.
\end{lemma}

\begin{proof}
  The set $\{x_i\}_{i \in \mathcal{I}}$ is poset-upper-triangular with respect to $<$ so
  it is also poset-upper-triangular with respect to any total ordering
  $\prec$ compatible with $<$.  This is Condition (i) of Proposition
  \ref{proposition:generate}.  

  To prove Condition (ii), we show that $x_j(j)$ generates
  $\mathcal{V}^j(j)$ as an $R$-module for each $j$.  We first claim
  that each $\mathcal{V}^j$ is generated by $\{x_i\}_{j \prec i
    \textup{ or } j=i}$. We proceed by induction. For the base case,
  let $j_0$ be the minimal element of $\mathcal{I}$ with respect to
  $\prec$. Then the set $\{i: j_0 \prec i \textup{ or } j_0=i\}$ is
  all of $\mathcal{I}$ and the claim trivially holds. Now suppose that
  the claim holds for all $i$ with $i \prec j$ and that $i'$ is
  maximal in $\{i \prec j\}$. (As in the proof of
  Proposition~\ref{proposition:generate}, a maximal $i'$ exists
  because $\{i: i \prec j\}$ is finite by the hypotheses on
  $\mathcal{I}$.) Let \(x \in \mathcal{V}^j \subseteq
  \mathcal{V}^{i'}.\) By the inductive hypothesis,
\[
x = c_{i'} x_{i'} + \sum_{i: i' \prec i} c_i x_i =
c_{i'} x_{i'} + \sum_{i: i=j \textup{ or } j \prec i} c_i x_i.
\]
The set is poset-upper-triangular so  $x_i(i')=0$ for all $i$ with $i' \prec i$.  Evaluating $x$ at $i'$ yields \(x(i') = c_{i'}
x_{i'}(i')\) which must equal $0$ since $x \in \mathcal{V}^j$. By assumption \(x_{i'}(i') \neq 0\) so $c_{i'}=0$.  This means $\mathcal{V}^j$ is generated by
$\{x_i\}_{j \prec i \textup{ or } j=i}$ as desired.

Finally, evaluation at $j$ yields
\(x(j)=c_j x_j(j).\) The element $x\in \mathcal{V}^j$ was chosen arbitrarily, so we conclude
$\mathcal{V}^j(j)$ is generated by the single element
$x_j(j)$ as claimed.  
\end{proof}

\subsection{Bases for submodules of products}\label{subsection:images of modules}

In this section we present the central problem of the manuscript,
stated in purely algebraic and combinatorial language. Its geometric
manifestation is reserved until
Section~\ref{sec:GKM-geometry}. We also explain the core difficulty
in addressing the problem, which motivates the poset pinball
game introduced in Section~\ref{sec:pinball}.

Let $\mathcal{I}$ be a countable
graded poset with each rank finite.  Let $\mathcal{J}$ be a subset
of $\mathcal{I}$ with the partial order induced from
$\mathcal{I}$. Let $R$ and $R'$ be integral domains, $M_i$ a
torsion-free $R$-module for each \(i \in \mathcal{I}\), and $M'_j$ a
torsion-free $R'$-module for each \(j \in \mathcal{J}.\) 
Let 
\(M \subseteq \prod_{i \in \mathcal{I}}M_i\) and \(M' \subseteq
\prod_{i \in \mathcal{J}} M'_i\) be $R$- and $R'$-submodules,
respectively, of the given products. Suppose \(\gamma: R \to R'\) is a
ring homomorphism and that \(\phi_i: M_i \to M'_i\) are surjective additive homomorphisms  for each \(i \in
\mathcal{J}\)
satisfying
  \begin{equation}
    \label{eq:module-twistedring}
    \phi(rm) = \gamma(r) \phi(m) \quad \textup{ for all } r \in R, i \in \mathcal{I}, \textup{ and } m \in M. 
  \end{equation}
We also assume the homomorphism $\prod_{i \in \mathcal{I}} \phi_i:
\prod_{i \in \mathcal{I}} M_i \rightarrow \prod_{i \in
  \mathcal{J}}M_i'$ restricts to a surjection $\phi: M \to M'$ so 
that the diagram 
\begin{equation}\label{eq:M-Mprime-diagram}
\xymatrix{
M \ar[d]_{\phi} \ar[r] & \prod_{i \in \mathcal{I}} M_i
\ar[d]^{\prod_{i \in \mathcal{J}} \phi_j} \\
M' \ar[r]& \prod_{i \in \mathcal{J}} M'_i
}
\end{equation}
commutes, where the right vertical map is understood to be $0$ on the components $M_i$ for \(i \not \in \mathcal{J}.\) 

Now suppose $\mathcal{B} = \{x_\alpha\}$ is a
poset-upper-triangular basis of $M$.  (In many examples,
such bases exist because of extra geometric structure; see
Section~\ref{sec:GKM-geometry}.) Poset-upper-triangularity
implies that the elements $x_\alpha$ have convenient vanishing
properties when viewed in the product $\prod_{i \in \mathcal{I}}
M_i$. The homomorphism $\phi$ is surjective, so there exists a subset of the image $\phi(\mathcal{B})
\subseteq M'$ which generates $M'$. 
We may then ask the following question (see
Question~\ref{question:main} for the geometric version). 

\begin{question}\label{question:main-algebraic}
Is it 
possible to obtain 
a poset-upper-triangular basis for $M' \subseteq \prod_{i \in
  \mathcal{J}} M'_i$ with respect to the
induced partial order on $\mathcal{J} \subseteq \mathcal{I}$ from a
subset of 
$\phi(\mathcal{B})$?
\end{question}

This turns out to be a difficult problem. The fundamental
obstacle is that the maps in~\eqref{eq:M-Mprime-diagram} do not necessarily
behave well with respect to poset-flow-ups. More precisely, the
intersection $\mathcal{L}_{\mathcal{I}}(i) \cap \mathcal{J}$ of a
principal order filter $\mathcal{L}_{\mathcal{I}}(i)$ with the subset
$\mathcal{J}$ is \emph{not} necessarily a
principal order filter of $\mathcal{J}$ when $i \not \in \mathcal{J}$. As a consequence, the images
$\phi(\mathcal{B}) = \{ \phi(x_\alpha)\}$ need not even be poset-flow-up elements in $M'$ in the sense of
Definition~\ref{def:poset-flowup}.
This means that poset-upper-triangularity with respect to
$(\mathcal{I}, <)$ does not immediately translate via $\phi$ to
poset-upper-triangularity with respect to $(\mathcal{J},
<)$.  In the next section, we introduce the combinatorial game of 
\textbf{poset pinball}, which was 
created to address these difficulties.

\section{Poset pinball: a combinatorial game on directed
  graphs}\label{sec:pinball}

\subsection{The poset pinball game}\label{subsec:pinball}

We now introduce a non-deterministic game which we call \textbf{poset
  pinball}, by analogy with pinball arcade games, which involve
dropping balls on a tilted board.  The game can be understood
and played independently of the considerations in the previous
section; however, the game was designed to address the algebraic
difficulties outlined in Section~\ref{subsection:images of modules}
(and discussed from a geometric perspective in 
Section~\ref{sec:GKM-geometry}).  We present several variants of poset
pinball; which flavor one plays depends on the geometric, algebraic,
or combinatorial context. Examples are in
Section~\ref{subsec:pinball-examples}.  In Section~\ref{subsec:ideals}, we discuss principal order ideals and the role they play in poset pinball.

We begin with the basic structure of the game, common to all
variants. Henceforth we assume that the poset $\mathcal{I}$ is
finite. 

\medskip
\noindent \textbf{Poset pinball rules and terminology:}

\begin{enumerate}
\item Let $(\mathcal{I}, <)$ be a finite partially ordered set.  We
  identify $\mathcal{I}$ with its Hasse diagram, so we think of
  $\mathcal{I}$ as a directed acyclic graph with vertices the elements
  of $\mathcal{I}$ and with a directed edge from $i$ to $i'$ precisely
  when $i$ covers $i'$ with respect to the partial order.  We denote
  an edge from $i$ to $i'$ by $i \mapsto i'$.
\item The graph $\mathcal{I}$ is the 
  \textbf{pinball board}, or simply the \textbf{board}. The vertices 
  are called \textbf{pinball slots}, or simply \textbf{slots}. At most one pinball
  can occupy a slot at any time. 
\item Let $\mathcal{J}$ be a fixed subset of $\mathcal{I}$. We call
  $\mathcal{J}$ an \textbf{initial
    subset}. Note that $\mathcal{J}$ inherits a partial order from
  $\mathcal{I}$. 
\item We place \textbf{pinballs} at the initial subset, i.e., for each
  element \(j \in \mathcal{J}\), we place a pinball at the slot
  corresponding to $j$. 
\item The directed edges of the graph $\mathcal{I}$ are called 
  \textbf{pinball slides}, or simply \textbf{slides}. When released, a pinball may
   \textbf{roll down} along a slide, in the direction determined by the
  directed edge. Specifically, if $i \mapsto i'$ is an edge, a pinball at
  slot $i$ may roll down to the slot $i'$. 
\item During the game, we occasionally place 
  \textbf{walls} across some slides. 
  A wall across a slide \textit{prevents} 
  a pinball from rolling down that edge (slide). The initial board has
  no walls. A wall is never
  removed once it has been placed.  
\item Fix a \textit{total order} $\prec$ on the initial subset
  $\mathcal{J}$ subordinate to the induced partial order on
  $\mathcal{J}$. We write \(\mathcal{J} = \{j_1 \prec j_2 \prec
  \cdots \prec j_{|\mathcal{J}|}\}\) with respect to this total
  order. 
\item We now define the procedure for allowing a pinball to \textbf{roll down
  (along slides)}. 
  Suppose a pinball is at slot $i \in
  \mathcal{I}$. Consider the set of downward-pointing edges with $i$
  as top vertex. The pinball at slot $i$ is allowed to roll
  down to $i'$  as long as 
there is no wall across the slide $i \mapsto i'$. 
Hence we consider 
\begin{equation}\label{eq:roll-step}
\left\{i' \in \mathcal{I}: \begin{array}{c}
\textup{ there exists an edge } i \mapsto i', \textup{ and } \\
\textup{there is no wall across } i \mapsto i'
\end{array}
\right\}.
\end{equation}
Choose an arbitrary element $i'$ in the
set in~\eqref{eq:roll-step} and move the pinball to slot $i'$. We
refer to this as \textbf{rolling along the slide $i \mapsto i'$.} Repeat the above process using the new slot $i'$ in the role of $i$
above, and continue in this manner. We say that the pinball can \textbf{roll no further} if at any
stage the set in~\eqref{eq:roll-step} is empty, namely there are no lower
available slots. When a pinball starting at slot $i$ has rolled down successive slides
until it can roll no further, the final
slot at which the pinball rests is called the \textbf{rolldown of $i$} and denoted $roll(i)$. 
We refer to this process of assocating to $i$ its rolldown
$roll(i)$ as \textbf{rolling (or dropping) the pinball}. 
This procedure  is \emph{not} deterministic because of the choices made when rolling along each slide (just like  real-life pinball!).
Note also that the rolldown $roll(j)$ of a pinball which was
originally at a slot $j \in \mathcal{J}$ might
\emph{not} be an element of $\mathcal{J}$. 

\item We drop pinballs successively according to the total order
  $\prec$ on $\mathcal{J}$. 
  Hence we first drop the pinball from slot $j_1$
  as described above. 
 For every 
  \(k = 1, 2, \ldots,  |\mathcal{J}|\), 
  after rolling the $k$-th pinball, we may place more walls along
  the slides of the board.  Each version of pinball has a separate set
  of rules for placing walls; details 
 for each variant are given in the description below. 
Once the first
  $k-1$ pinballs are dropped, we drop the pinball at $j_k$ and continue until all $|\mathcal{J}|$ pinballs are dropped. 

\item Fix the board $\mathcal{I}$, the 
  initial subset $\mathcal{J}$, the choice of total order $\prec$, and a particular outcome of poset pinball, which we write as $\{(j, roll(j)): j \in \mathcal{J}\}$.  Then we denote by $\mathcal{R}(\mathcal{I},
  \mathcal{J}, \prec)$ the set of slots in ${\mathcal{I}}$ occupied by the
  rolldown elements, i.e. 
\[
\mathcal{R}(\mathcal{I}, \mathcal{J}, \prec) := \{ roll(j): j \in
\mathcal{J}\} \subseteq \mathcal{I}.
\]
We call $\mathcal{R}(\mathcal{I},
\mathcal{J}, \prec)$ the \textbf{rolldown set} for the
given outcome of pinball. We also denote by $\mathcal{R}_k(\mathcal{I}, \mathcal{J},
\prec)$ the set of rolldown elements for the first $k$
pinballs, i.e. 
\[
\mathcal{R}_k(\mathcal{I}, \mathcal{J}, \prec) := \{ roll(j_\ell):
j_\ell \in \mathcal{J}, 1 \leq \ell \leq k \} = \{roll(j_1),
roll(j_2), \ldots, roll(j_k) \} \subseteq \mathcal{I}.
\]
We refer to $\mathcal{R}_k(\mathcal{I}, \mathcal{J}, \prec)$ as the
\textbf{rolldown set up to step $k$}. 
We emphasize that since pinball is \emph{not} deterministic, the
sets $\mathcal{R}(\mathcal{I}, \mathcal{J}, \prec)$ and
$\mathcal{R}_k(\mathcal{I}, \mathcal{J}, \prec)$ might \emph{not}
be uniquely determined by $\mathcal{I}, \mathcal{J}$ and
$\prec$.
\end{enumerate}

The different versions of poset
pinball are distinguished by how the walls are placed after each
pinball in the initial subset rolls down. 
We now describe these variants of poset pinball.

\medskip

\noindent \textbf{Basic pinball.} In this version, the walls are placed as
  follows. Let \(k =1, 2, \ldots, |\mathcal{J}|.\) 
  Suppose the $k$-th pinball $j_k$ has been dropped. We then place a wall across every edge of the form $i \mapsto roll(j_k)$. 
Hence the walls in basic pinball simply 
enforce the rule that at most one pinball may occupy a given
slot at any time.

\medskip

\noindent \textbf{Upper-triangular pinball}. 
This version takes into account the partial order structure on 
$\mathcal{I}$. 
Let \(k =1, 2, \ldots, |\mathcal{J}|.\) Suppose the
    $k$-th pinball $j_k$ has been dropped. We then place a wall across: 
\begin{itemize} 
\item every edge of the form $i \mapsto roll(j_k)$ for \(i \in
  \mathcal{I}\), and 
\item every edge of the form $i \mapsto i'$ for \(i \in \mathcal{I}, i' \in
  \mathcal{L}_{j_k}.\) 
\end{itemize} 
The rules of upper-triangular pinball ensure that for each \(j \in
\mathcal{J}\), the element $j$ has a unique maximal rolldown in its
principal order ideal, and that maximal rolldown is $roll(j)$.

\medskip

\noindent \textbf{Betti pinball.}  This version of pinball is
motivated by the geometric applications discussed later, in which we
allow the Betti numbers of an underlying topological space to impose
additional constraints on the pinball game.  Here we assume that
$\mathcal{I}$ is a finite graded poset with rank function \(\rho:
\mathcal{I} \to \N.\) We also assume \(\mathbf{b} = (b_0, b_1, \ldots,
b_n)\) is a sequence of nonnegative integers. (In geometric
applications, these $b_j$ are in fact the Betti numbers of a topological
space, so we refer to $\mathbf{b}$ as the {\em target Betti numbers}.)
Let \(k =1, 2, \ldots, |\mathcal{J}|.\) Suppose the $k$-th pinball
$j_k$ has been dropped. Then the walls are placed as follows:
\begin{itemize} 
\item Place a wall across any edge of the form $i \mapsto roll(j_k)$ for \(i \in
  \mathcal{I}.\)  
\item Let $\mathcal{R}_k(\mathcal{I}, \mathcal{J}, \prec)$ denote the
  rolldown set up to step $k$. Let \(j \in \{0, 1, \ldots, n\}\) and
  suppose that 
\begin{equation}\label{eq:bettifull}
b_j = | \{ i \in \mathcal{R}_k(\mathcal{I}, \mathcal{J}, \prec):
\rho(i) = j \} |,
\end{equation}
i.e. there are exactly $b_j$ rolldown elements of rank
$j$ at step $k$. 
Then place walls across every edge
of the form 
$i \mapsto i'$ where $\deg(i') = j$. 
\end{itemize} 
These rules ensure that the number of
rolldown elements of rank $j$ do not exceed the given
target $b_j$. 

\medskip

\noindent We say that a game of Betti pinball is
\textbf{successful} if, after all the pinballs in the initial subset
$\mathcal{J}$ are dropped, there are precisely $b_j$
rolldowns of rank $j$ for each $j$.  In other words, after a successful game of Betti pinball, the ranks of
the elements of the rolldown set $\mathcal{R}(\mathcal{I},
\mathcal{J}, \prec)$ precisely reflect the target Betti numbers
$\mathbf{b}$.

\medskip

\noindent \textbf{Upper-triangular Betti
  pinball}. This version adds both the walls for
upper-triangular pinball and those for Betti pinball at each pinball step. We leave it to the
reader to write the rules. 
As in Betti pinball, we assume we are given a graded poset $\mathcal{I}$  and 
the data of {\em target Betti numbers} 
$\mathbf{b} = (b_0, b_1, b_2, \ldots, b_n)$. 
Also as in Betti pinball, we say that a game of upper-triangular Betti pinball is \textbf{successful} if the
ranks of the rolldowns precisely reflect the target Betti numbers.

\subsection{Playing pinball: examples}\label{subsec:pinball-examples}

Here we illustrate our poset pinball game with two concrete examples. 
In both cases, the ambient graded poset $\mathcal{I}$ is the
symmetric group $S_4$ equipped with the usual Bruhat order.  We take the rank of a permutation
$w$ to be the standard \textbf{Bruhat length} of $w$.
For simplicity, we do not draw the entire graph
of $\mathcal{I}$ in the figures below, but only those vertices and edges relevant
in the game.  We
label each vertex by its corresponding permutation, factored into simple
transpositions. 
The graph is drawn so that
all directed edges point towards the bottom of the page, so 
the minimal permutation $e$ is at the bottom of the figure.
The initial subset $\mathcal{J}$ is indicated by the circled vertices,
and the final rolldown set is indicated by the vertices with squares
around them.
Each example is accompanied by a table recording each
step of the poset pinball game as it was played.

\begin{example}
  In our first example, the rolldown set is in fact unique for basic,
  upper-triangular, or Betti pinball with $\mathbf{b} = (1,3)$. 
  In particular, the partial order induced on $\mathcal{J}$ by
  $\mathcal{I}$ is already a total order, so there is a unique total
  order with respect to which to play pinball. Notice that the rolldown set
  $\mathcal{R}(\mathcal{J}, \mathcal{I}, \prec)$ in this example is a
  union of principal order ideals; we explore this phenomenon further
  in Section~\ref{subsec:ideals}. 

The initial subset is \(\mathcal{J} = \{e, s_3, s_3s_2,
  s_3s_2s_1\}.\) 
The final drop-down set is
  \(\mathcal{R}(\mathcal{I}, \mathcal{J}) = \{e, s_3, s_2, s_1\}.\)  (The reader may wish to explore how the game changes if the Betti numbers are $\mathbf{b} =(1,2,1)$.)

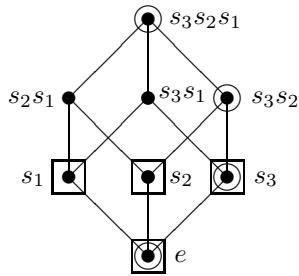
\begin{figure}[h]
\begin{picture}(90,90)(0,0)
\multiput(45,0)(0,30){4}{\circle*{5}}
\multiput(15,30)(30,0){3}{\circle*{5}}
\multiput(15,60)(30,0){3}{\circle*{5}}
\put(45,0){\line(-1,1){30}}
\put(45,0){\line(0,1){30}}
\put(45,0){\line(1,1){30}}

\put(45,90){\line(-1,-1){30}}
\put(45,90){\line(0,-1){30}}
\put(45,90){\line(1,-1){30}}

\put(15,30){\line(0,1){30}}
\put(15,30){\line(1,1){30}}

\put(45,30){\line(1,1){30}}
\put(45,30){\line(-1,1){30}}

\put(75,30){\line(0,1){30}}
\put(75,30){\line(-1,1){30}}

\put(45,90){\circle{10}}
\put(75,60){\circle{10}}
\put(75,30){\circle{10}}
\put(45,0){\circle{10}}

\put(39,-6){\framebox(12,12)}
\put(39,24){\framebox(12,12)}
\put(9,24){\framebox(12,12)}
\put(69,24){\framebox(12,12)}

\put(55,-2){$e$}
\put(85,28){$s_3$}
\put(53,28){$s_2$}
\put(-3,28){$s_1$}
\put(84,58){$s_3s_2$}
\put(49,60){$s_3s_1$}
\put(-8,58){$s_2s_1$}
\put(53,88){$s_3s_2s_1$}

\end{picture}
\caption{Example of basic pinball.}\label{fig:SpringerN=4,YoungX=3,1}
\end{figure}

\renewcommand{\arraystretch}{1.3}
\begin{equation}\label{eq:SpringerN=4,YoungX=3,1-pinball}
\begin{array}{c||c|c|}
\mbox{pinball step} & w_k & v_k  \\ \hline \hline
1 & w_1 = e  & v_1 = e \\ \hline
2 & w_2 = s_3 & v_2 = s_3  \\ \hline
3 & w_3 = s_3 s_2 & v_3 = s_2  \\ \hline
4 & w_4 = s_3 s_2 s_1 & v_4 = s_1  \\ \hline
\end{array}
\end{equation}

\end{example}

\begin{example}\label{example:RegNilpN=4,h=3344}

In this example, we play Betti pinball with target Betti numbers
$\mathbf{b} = (1,3,4,3,1)$.  
 The final drop-down set $\mathcal{R}(\mathcal{I}, \mathcal{J})$
is indicated by the squared vertices. 
Dotted lines indicate a path in the partial order, with some intermediate vertices
  omitted for visual simplicity.  Also for visual simplicity, not all edges in the poset are drawn above rank $2$.
In this example, the drop-down set $\mathcal{R}(\mathcal{J},
\mathcal{I}, \prec)$ is \emph{not} a union of principal order ideals
because of the constraints imposed by the target Betti
numbers. 

The rolldown set in Figure~\ref{fig:RegNilpN=4,h=3344-pinball} could
also be the outcome of a game of upper-triangular Betti pinball.
However, if we instead chose the rolldown $roll(s_1s_2s_3s_1s_2s_1) =
s_1s_2s_3s_2$ then we would obtain a successful outcome of Betti
pinball that is not a successful outcome of upper-triangular Betti
pinball.

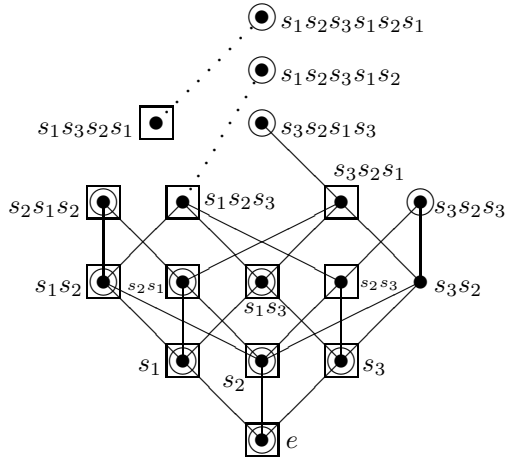
\begin{figure}[h]
\begin{picture}(150,160)(0,0)
\put(75,0){\circle*{5}}
\multiput(45,30)(30,0){3}{\circle*{5}}
\multiput(15,60)(30,0){5}{\circle*{5}}
\multiput(15,90)(30,0){2}{\circle*{5}}
\multiput(105,90)(30,0){2}{\circle*{5}}
\multiput(75,120)(0,20){3}{\circle*{5}}
\put(35,120){\circle*{5}}

\put(75,0){\line(-1,1){30}}
\put(75,0){\line(0,1){30}}
\put(75,0){\line(1,1){30}}

\put(45,30){\line(-1,1){30}}
\put(45,30){\line(0,1){30}}
\put(45,30){\line(1,1){30}}
\put(75,30){\line(-1,1){30}}
\put(75,30){\line(-2,1){60}}
\put(75,30){\line(1,1){30}}
\put(75,30){\line(2,1){60}}
\put(105,30){\line(-1,1){30}}
\put(105,30){\line(0,1){30}}
\put(105,30){\line(1,1){30}}

\put(15,90){\line(0,-1){30}}
\put(15,90){\line(1,-1){30}}
\put(135,90){\line(0,-1){30}}
\put(135,90){\line(-1,-1){30}}
\put(45,90){\line(-1,-1){30}}
\put(45,90){\line(1,-1){30}}
\put(45,90){\line(2,-1){60}}
\put(105,90){\line(1,-1){30}}
\put(105,90){\line(-1,-1){30}}
\put(105,90){\line(-2,-1){60}}

\multiput(75,160)(-4,-4){10}{\circle*{1}}
\multiput(75,140)(-3,-5){10}{\circle*{1}}
\put(75,120){\line(1,-1){30}}

\put(75,0){\circle{10}}
\put(45,30){\circle{10}}
\put(75,30){\circle{10}}
\put(105,30){\circle{10}}
\put(15,60){\circle{10}}
\put(45,60){\circle{10}}
\put(75,60){\circle{10}}
\put(15,90){\circle{10}}
\put(135,90){\circle{10}}
\multiput(75,120)(0,20){3}{\circle{10}}

\put(69,-6){\framebox(12,12)}
\multiput(39,24)(30,0){3}{\framebox(12,12)}
\multiput(9,54)(30,0){4}{\framebox(12,12)}
\put(9,84){\framebox(12,12)}
\put(39,84){\framebox(12,12)}
\put(99,84){\framebox(12,12)}
\put(29,114){\framebox(12,12)}

\put(84,-3){$e$}
\put(113,26){$s_3$}
\put(60,20){$s_2$}
\put(28,26){$s_1$}
\put(-11,56){$s_1s_2$}
\put(140,56){$s_3s_2$}
\put(24,57){\tiny $s_2s_1$}
\put(112,58){\tiny $s_2s_3$}
\put(68,48){\small $s_1s_3$}
\put(-21,86){$s_2s_1s_2$}
\put(140,86){$s_3s_2s_3$}
\put(53,89){$s_1s_2s_3$}
\put(102,100){$s_3s_2s_1$}
\put(-10,116){$s_1s_3s_2s_1$}
\put(82,116){$s_3s_2s_1s_3$}
\put(82,136){$s_1s_2s_3s_1s_2$}
\put(82,156){$s_1s_2s_3s_1s_2s_1$}

\end{picture}
\caption{An example of Betti pinball.}\label{fig:RegNilpN=4,h=3344-pinball}
\end{figure}

\renewcommand{\arraystretch}{1.3}
\begin{equation}\label{eq:RegNilpN=4,h=3344-pinball}
\begin{array}{c||c|c|}
\mbox{pinball step} & w_k & v_k  \\ \hline \hline
1 & w_1 = e = [1,2,3,4] & v_1 = e =[1,2,3,4]\\ \hline
2 & w_2 = s_3 =[1,2,4,3] & v_2 = s_3 = [1,2,4,3] \\ \hline
3 & w_3 = s_2 = [1, 3, 2, 4] &  v_3 = s_2 = [1, 3, 2, 4] \\ \hline
4 & w_4 = s_1 = [2, 1, 3, 4] &  v_4 = s_1 = [2, 1, 3, 4] \\ \hline 
5 & w_5 = s_1 s_3 = s_3 s_1 = [2, 1, 4, 3] &  v_5 = s_1 s_3 = s_3 s_1
= [2, 1, 4, 3] \\ \hline 
6 & w_6 = s_1 s_2 = [2, 3, 1, 4] & v_6 = s_1 s_2 = [2, 3, 1, 4] \\ \hline 
7 & w_7 = s_2 s_1 = [3, 1, 2, 4] &  v_7 = s_2 s_1 = [3, 1, 2, 4] \\ \hline 
8 & w_8 = s_3 s_2 s_3 = [1, 4, 3, 2] & v_8 = s_2 s_3 = [1, 3, 4, 2] \\ \hline
9 & w_9 = s_2 s_1 s_2 = [3, 2, 1, 4] & v_9 = s_2 s_1 s_2 = [3, 2, 1, 4] \\ \hline 
10 & w_{10} = s_3 s_2 s_1 s_3 = [4, 1, 3, 2] & v_{10} = s_3 s_2 s_1 = [4, 1, 2, 3] \\ \hline 
11 & w_{11} = s_1 s_2 s_3 s_1 s_2 = [3, 4, 2, 1] &  v_{11} = s_1 s_2
s_3 = [2, 3, 4, 1] \\ \hline 
12 & w_{12} = s_1 s_2 s_3 s_1 s_2 s_1 = [4, 3, 2, 1] & v_{12} = s_1
s_3 s_2 s_1 = [4, 2, 1, 3] \\ \hline 
\end{array}
\end{equation}
\end{example}

\subsection{Principal order ideals and poset pinball}\label{subsec:ideals}

In this section, we briefly explore the role played by  
principal order ideals in poset pinball. We are motivated
by our geometric applications, in which principal order ideals can correspond
naturally to subvarieties in an ambient variety
(cf. Remark~\ref{remark:schubert}). 

We begin with a simple statement about basic pinball.

\begin{proposition}
 Let $(\mathcal{I}, <)$ be a finite poset and let 
 $\mathcal{J} \subseteq \mathcal{I}$ be a subset. Let $\prec$ be
 a total ordering on $\mathcal{J}$ compatible with the partial order $<$ induced from $\mathcal{I}$. 
  Suppose \(\mathcal{R}(\mathcal{I}, \mathcal{J}, \prec) \subseteq
  \mathcal{I}\) is the rolldown set from a game of basic pinball
  played with board $\mathcal{I}$, initial set $\mathcal{J}$, and total order $\prec$. 
   Then $\mathcal{R}(\mathcal{I},\mathcal{J}, \prec)$ is
  a union of principal order ideals of $\mathcal{I}$. 
\end{proposition}

\begin{proof}
  A subset \(\mathcal{K} \subseteq \mathcal{I}\) is a union of
  principal order ideals exactly if for all \(i \in \mathcal{K}\), the
  principal order ideal $\mathcal{L}_i$ is entirely contained in
  $\mathcal{K}$.  Let \(\mathcal{J} = \{j_1, j_2, \ldots,
  j_{|\mathcal{J}|}\}\) be the totally-ordered initial subset.  We
  will induct on $k = 1, 2, \ldots, |\mathcal{J}|$ to show that each
  $\mathcal{R}_k(\mathcal{I}, \mathcal{J}, \prec)$ is a union of
  principal order ideals.  In the base case, the rolldown $roll(j_1)$
  associated to $j_1$ must be minimal in $\mathcal{I}$, by definition
  of basic pinball.  A minimal element $\{roll(j_1)\}$ is a principal
  order ideal, so the claim holds for $k=1$. Now assume that
  $\mathcal{R}_{k-1}(\mathcal{I}, \mathcal{J}, \prec)$ is a union of
  principal order ideals. Let \(roll(j_k) \leq j_k\) be a rolldown of
  $j_k$. If $i < roll(j_k)$ then $i$ must be in the rolldown set
  $\mathcal{R}_{k-1}(\mathcal{I}, \mathcal{J}, \prec)$ by the rules of
  basic pinball.  Hence $\mathcal{L}_i$ is contained in
 \[\mathcal{R}_k(\mathcal{I}, \mathcal{J}, \prec) = \mathcal{R}_{k-1}(\mathcal{I}, \mathcal{J}, \prec) \cup \{roll(j_k)\}.\]
The rolldown set $\mathcal{R}_{k-1}(\mathcal{I}, \mathcal{J}, \prec)$ is itself a union of principal order ideals by the inductive hypothesis.  Hence  $\mathcal{R}_k(\mathcal{I}, \mathcal{J}, \prec)$ is also a union of
 principal order ideals. The case \(k = |\mathcal{J}|\) proves the
 proposition. 
\end{proof}

The previous proposition only applies to basic pinball.  In
upper-triangular or Betti pinball, the additional walls placed during
the game imply that the resulting rolldown set
$\mathcal{R}(\mathcal{I}, \mathcal{J}, \prec)$ may not be a union of
principal order ideals.  Indeed, Example~\ref{example:RegNilpN=4,h=3344} is an
instance of Betti pinball in which the associated rolldown set
$\mathcal{R}(\mathcal{I}, \mathcal{J}, \prec)$ is not a union of
principal order ideals. On the other hand, in many examples (like the
example in the Introduction or Example~\ref{example:SprN=4,X=211}),
Betti pinball does produce rolldown sets which are unions of principal
order ideals. This leads us to ask the following.

\begin{question}\label{question:ideals}
  What are combinatorial conditions on $\mathcal{I}, \mathcal{J},
  \prec$, and (in the case of Betti pinball) the target Betti numbers
  $(b_0, \ldots, b_n)$ which guarantee that outcomes of
  upper-triangular or Betti pinball are unions of principal order
  ideals?
\end{question} 

A concrete answer to this question would yield new perspectives
on the geometric problems, such as computing cohomology rings or
Betti numbers, that motivate our pinball game.

\begin{remark}\label{remark:schubert}
One situation in which an answer to Question \ref{question:ideals} is
straightforward is when the
  initial subset $\mathcal{J}$ is itself the principal order ideal
  $\mathcal{L}_i$ of some element \(i \in \mathcal{I}.\)  In this case,
  a successful outcome of {\em any} version of pinball is 
  the rolldown set $\mathcal{R}(\mathcal{I},
  \mathcal{J}, \prec)$ that equals the original $\mathcal{J} =
  \mathcal{L}_i$. This situation 
  can arise naturally in geometric contexts. We give two examples.
\begin{enumerate}
\item The poset $\mathcal{I}$ is the Weyl group $W$
of a complex reductive algebraic group $G$, identified with the
$T$-fixed points of its flag variety $G/B$.  In this case, a principal
order ideal $\mathcal{L}_w$ of \(w \in W\) corresponds naturally to
the $T$-fixed points in a
Schubert subvariety of $G/B$.  
\item Let $X$ be a complex projective algebraic variety (possibly
  singular) equipped with a $T$-action that has isolated $T$-fixed
  points $X^T$.  Choose a one-parameter subgroup $S: \mathbb{G}_m
  \rightarrow T$ with $X^S = X^T$.  Then $X$ is partitioned into
  locally closed subsets $X_p$ defined by
  \[X_p := \{x \in X: \lim_{z \mapsto 0} S(z) \cdot x = p\}.\] The
  disjoint union $\bigsqcup X_p$ is called a {\em Bialynicki-Birula
    decomposition} of $X$ \cite{BiaBir76}.  Moreover Knutson states that $X^T$ can be
  given a poset structure by taking the transitive closure of the rule that
  $p \leq q$ when $p \in \overline{X^q}$ \cite{Knu08} 
The principal order ideal
  $\mathcal{L}_p$ of $p \in X^T$ then corresponds naturally to the
  $T$-fixed points in 
  $\overline{X_p}$. 
\end{enumerate}
\end{remark}

\section{Poset pinball for GKM-compatible subspaces}\label{sec:GKM-geometry}

\subsection{Background: GKM theory in equivariant cohomology}\label{subsec:background-GKM}

The algebraic questions discussed in the previous sections arise
naturally in equivariant algebraic topology. Suppose $G$ is a
topological group and $X$ is a topological space with a continuous
$G$-action. Let $\pt$ denote the topological space consisting of one
point, equipped with the trivial $G$-action, and let $E^*_G$ denote a
generalized equivariant cohomology theory with a commutative cup
product. (Examples include Borel-equivariant cohomology $H^*_G(-;
\lie{r})$ for various coefficient rings $\lie{r}$, topological
equivariant $K$-theory in the sense of Atiyah and Segal, and
equivariant cobordism; cf. \cite[Chapter XIII]{May96}.)  Then
$E^*_G(\pt)$ is a commutative ring, and $E^*_G(X)$ is naturally an
$E^*_G(\pt)$-module for any $G$-space $X$ via the map induced on
cohomology by the $G$-equivariant map $X \rightarrow \pt$.

We work in a situation in which $E^*_G(X)$ has a well-studied
combinatorial description, often called ``GKM theory'' due to an
influential manuscript of Goresky-Kottwitz-MacPherson \cite{GKM}. We
present one of many variations and generalizations of GKM theory in
the literature (see 
\cite{HHH05} and references therein).

For the purposes of this manuscript only, we say that \textbf{the
  GKM package} holds for a $G$-space $X$ when the 
following statements hold. 
Let $G$ and $X$ be as above. 

\begin{itemize}
\item We assume that the $G$-fixed
set $X^G$ consists of countably many isolated points, i.e. \(X^G \cong \bigcup_{i \in
  \mathcal{I}} F_i\) with $F_i \cong \pt$ for all \(i \in
\mathcal{I}.\) 
\item We assume that the indexing set $\mathcal{I}$ for the fixed
  points $\bigcup_{i \in \mathcal{I}} F_i$ may
be equipped with a graded partial order such that there are
only finitely many elements of each rank. 
\item We assume that $X$ is a stratified $G$-space \(X = \cup_{i \in
    \mathcal{I}} X_i\) with $F_i \in X_i$ for each \(i \in
  \mathcal{I}\) and that the cohomology \(E^*_G(X) = \underleftarrow{\lim} \hsm
  E^*_G(X_i).\) 
\item We assume the restriction map 
\begin{equation}\label{eq:iota-star}
\iota^*: E^*_G(X) \to \prod_{i \in {\mathcal I}} E^*_G(F_i) \cong
\prod_{i \in \mathcal{I}} E^*_G(\pt) 
\end{equation}
is injective.
\item We assume there exist 
certain nonzero equivariant cohomology 
classes $e_{ij} \in E^*_G(\pt)$ 
  satisfying:
\begin{itemize} 
\item if $e_{ij} \neq 1$ 
  then $i$ and $j$ are comparable in $\mathcal{I}$,
\item if there is a covering relation between $i>j$ in $\mathcal{I}$
  then $e_{ij} \neq 1$
\end{itemize}
such that the image of $\iota^*$ in~\eqref{eq:iota-star} is precisely 
\begin{equation}\label{eq:GKM-isolated}
  \image(\iota^*) = \left\{ x \in \prod_{i \in {\mathcal I}} E_G^*(\pt) \, \bigg\vert \, e_{ij} \big\vert  \left(x(i) - x(j) \right) \,\, \textup{ for all } j
    < i \textup{ in the partial order}. \right\}
\end{equation}
\item We assume there exists a module basis $\{x_i\}_{i \in \mathcal{I}}$ for
  $E^*_G(X)$ that is indexed by the (isolated) fixed points $F_i$, where 
 the \(x_i \in E^*_G(X)\) are equivariant cohomology classes
satisfying
\begin{equation}\label{eq:x_i at i'} 
x_i(i') = 0 \, \mbox{ for} \, i' \not \geq i,
\end{equation}
and 
\begin{equation}\label{eq:generates ideal} 
x_i(i) \, \mbox{generates the ideal} \, e_i E^*_G(\pt)
\end{equation}
where $e_i := \prod_{j < i} e_{ij}$. 
(In this case \(\{x_i\}_{i \in {\mathcal I}}\) is a free
$E^*_G(\pt)$-module basis of $E^*_G(X)$ as proven in, e.g., \cite[Proposition
4.1]{HHH05}.)
\end{itemize} 

The vanishing condition in Equation~\eqref{eq:x_i at i'} says exactly
that the $x_i$ are poset-flow-ups, so the $\{x_i\}_{i \in
  \mathcal{I}}$ are in fact a poset-upper-triangular module basis by Proposition~\ref{proposition:independent}. 

For convenience, we often collect the information needed to determine
$E^*_G(X)$ using Equation~\eqref{eq:GKM-isolated} in a directed,
labeled graph called the {\bf GKM graph of $X$}.  The vertices of the
GKM graph are the fixed points $F_i$, or equivalently the elements of
the poset $\mathcal{I}$.  There is an edge between $F_i$ and $F_j$
exactly when $e_{ij} \neq 1$. If it exists, the edge between
$F_i$ and $F_j$ is labeled $e_{ij}$ and is directed from $F_i$ to
$F_j$ exactly when $i>j$.  Note that the GKM graph of $X$ contains the
Hasse diagram of the poset $\mathcal{I}$, but possibly includes
edges which are not poset covering relations. 

We now recall some situations in which the 
GKM package holds.

\begin{remark}\label{remark:GKM conditions} 
 \begin{enumerate}

 \item Let $G=T$ be a compact torus and let $E^*_G = H^*_T(-;\F)$ be
   Borel-equivariant cohomology with coefficients in a field $\F$ of
   characteristic zero. 
Let $(M, \omega, \Phi)$ be a compact
   Hamiltonian $T$-manifold with moment
   map $\Phi: M \to \t^*$.  Suppose that $M$ has finitely many
   (isolated) fixed points, and for every codimension $1$ subtorus $K
   \subseteq T$, each connected component of the fixed submanifold
   $M^K$ has (real) dimension less than or equal to $2$. Assume in
   addition that there are finitely many one-dimensional $T$-orbits in
   $M$ and that there exists a $T$-invariant Palais-Smale
   metric. Let $\Psi := \Phi^\xi$ denote a generic component of the
   moment map and consider its negative gradient flow with respect to
   the given Palais-Smale metric. Denote by $\lambda(p)$ the Morse
   index of $\Psi$ at a critical point
   $p \in M^T$. It is known \cite[Remark 4.3]{GolTol09} that $\Psi$ is
   \emph{index-increasing}, i.e. if
    \(p, q \in M^T\) and \(\Psi(p) < \Psi(q)\), then \(\lambda(p) <
    \lambda(q).\) We give the fixed point set $M^T$ a partial order by
    defining \(p < q\) precisely if  \(\Psi(p) <
    \Psi(q)\) and there exists a one-dimensional $T$-orbit $O$ whose closure contains $p$ and $q$.
    In particular \(p < q\) implies \(\lambda(p) <
    \lambda(q).\) (The Hasse diagram of this poset
     $M^T$ has an edge between $p$ and $q$ exactly when $p<q$ and $\lambda(q) = \lambda(p)+2$,
    i.e. the Morse index increases from $p$ to $q$ by precisely $2$.)
By standard equivariant symplectic geometry
    arguments, the Morse index is always even. Thus the 
    function $\rho: M^T \to \Z_{\geq 0}$ defined by \(\rho(p) := \frac{1}{2}
    \lambda(p)\) is a well-defined rank function, and the fact that $\Psi$ is index-increasing implies
    that $\rho$ gives $M^T$ the structure of a \emph{graded} poset.

    Moreover, in this situation it is also known that the
    map~\eqref{eq:iota-star} is injective and that the description of
    the image of $\iota^*$ given in~\eqref{eq:GKM-isolated} is valid
    \cite[Theorem 14.1 (9)]{GKM}. 
   Finally, in this situation, Goldin and Tolman prove that there exists a
    collection $\{x_p\}_{p \in M^T}$ of \emph{canonical
      classes} \cite[Definition 1.1]{GolTol09} in $H^*_T(M;\F)$ 
    indexed by the $T$-fixed points $M^T$.  These canonical classes give a module basis for
    $H^*_T(M;\F)$ satisfying Equations~\eqref{eq:x_i at i'}
    and~\eqref{eq:generates ideal} with respect to the partial
    order on $M^T$ defined above \cite[Proposition 4.4]{GolTol09}.  In particular, they give
    a poset-upper-triangular module basis with respect to the
    partial order on $M^T$.  Moreover the $x_p$ are 
    homogeneous classes of degree $2\rho(p)$. (In fact, their
    canonical classes form a module basis for $H^*_T(M;\Z)$ with
    integer coefficients, but we will not use that here.)

  \item Suppose $\mathcal{G}$ is a Kac-Moody group and $\mathcal{P}$ a
    parabolic subgroup with corresponding flag variety $X =
    \mathcal{G}/\mathcal{P}$.  For \(G = T_{\mathcal
      G}/Z(\mathcal{G})\) where $T_{\mathcal G}$ is the maximal torus
    of $\mathcal{G}$ and $Z(\mathcal{G})$ is its center, 
    Equations~\eqref{eq:iota-star} and~\eqref{eq:GKM-isolated} hold
    for many cases of $E^*_G$
    (see e.g. \cite{HHH05}).  The set of $T$-fixed points of
    $\mathcal{G}/\mathcal{P}$ may be identified with the quotient
    $W_{\mathcal{G}}/W_{\mathcal{P}}$ where $W_{\mathcal{G}},
    W_{\mathcal{P}}$ are the Weyl groups of $\mathcal{G}, \mathcal{P}$
    respectively. The index set $\mathcal{I}$ may be identified with
    this countable quotient and given a poset structure and rank function induced by the Bruhat
    order and Bruhat length on $W_{\mathcal{G}}$ respectively.  
    In the cases of Borel-equivariant
    cohomology $E^*_T=H^*_T(-)$ or equivariant $K$-theory $E^*_T =
    K^*_T$, the \textbf{equivariant Schubert classes} $\{\sigma_w\}_{w
      \in W_{\mathcal{G}}/W_{\mathcal{P}}}$ corresponding to the
    Schubert varieties in $\mathcal{G}/\mathcal{P}$ form a module
    basis satisfying Equations~\eqref{eq:x_i at i'}
    and~\eqref{eq:generates ideal} (see e.g. \cite{Kum02} and also
    \cite{Wil04, Wil06}). 
    Unlike the
    previous example, this one includes cases of infinite posets
    $\mathcal{I}$. Harada-Henriques-Holm give some explicit
    computations for $\Omega SU(2)$, an infinite-dimensional affine
    Grassmannian \cite{HHH05}.
 \end{enumerate}
\end{remark}

We close the section with a brief discussion of a more general class
of spaces for which poset-upper-triangular bases satisfying
conditions~\eqref{eq:x_i at i'} and~\eqref{eq:generates ideal} exist in
Borel-equivariant cohomology.

\begin{theorem}\label{theorem:poset-upper-triangular conditions}
  Let $S \cong S^1$ or $S \cong \mathbb{C}^*$ be a rank-one torus and
  let $X$ be an $S$-space satisfying the following conditions:
\begin{enumerate}
\item the set of $S$-fixed points $X^S$ is isolated;
\item the set $(X^S,<)$ can be equipped with a poset structure satisfying the conditions of Section~\ref{subsec:combinatorics};
\item the $S$-equivariant cohomology $H^*_S(X; \F)$ is a free
 $H^*_S(\pt;\F)$-module; and
\item the ring map \(\iota_X^*: H^*_S(X;\F) \to H^*_S(X^S;\F)\) induced
 by the inclusion \(\iota_X: X^S \into X\) is injective.
\end{enumerate}
Then the module $H^*_S(X;\F)$ has a poset-upper-triangular basis with
respect to any choice of total order compatible with the partial order on
$X^S$.
\end{theorem}

\begin{proof}
  The space $X$ has isolated $S$-fixed points and $H^*_S(X;\F)$ is a
  free $H^*_S(\pt;\F)$-module isomorphic to a submodule of
  $\bigoplus_{i \in X^S} H^*_S(\pt;\F)$.  The ring $H^*_S(\pt;\F)$ is
  isomorphic to the polynomial ring $\F[t]$ in one variable, which is
  a PID.  Hence the sets $\mathcal{V}^i(i)$ in
  Proposition~\ref{proposition:generate} are generated by at most one
  nonzero element, so the index sets $\mathcal{K}_i$ in
  Proposition~\ref{proposition:generate} have at most one element.
  The claim follows from Proposition~\ref{proposition:independent}.
\end{proof}

\begin{remark}
  Conditions (3) and (4) in
  Theorem~\ref{theorem:poset-upper-triangular conditions} will
  recur in this manuscript, for instance in the definition of
  GKM-compatibility.
\end{remark}

\begin{remark}
  Many natural topological spaces satisfy the conditions of
  Theorem~\ref{theorem:poset-upper-triangular conditions}, including
  many of those described in Remark~\ref{remark:GKM conditions}. The
  poset structure in Condition (2) of
  Theorem~\ref{theorem:poset-upper-triangular conditions}
is often
  induced from the $S$-action, as described in
  Remark~\ref{remark:schubert} for complex projective varieties
  equipped with a $T$-action and in
  Remark~\ref{remark:GKM conditions} for Hamiltonian $T$-manifolds.
  The existence of a 
  paving by complex affine cells, as is sometimes induced by a
  Bialynicki-Birula decomposition, implies Conditions (3) and (4).
\end{remark}

\subsection{Subspaces of GKM spaces}\label{subsec:subspace-setup}

GKM theory is a powerful combinatorial tool that can provide natural
and computationally convenient bases for $E^*_G(X)$. However, there
are many $G$-spaces for which the GKM package does not hold. This
brings us to the central geometric problem of this manuscript. More
specifically, we will describe a geometric framework within
which we propose to exploit the GKM theory on an ambient space $X$ in
order to analyze the equivariant geometry of a subspace $Y \subseteq
X$. We carried out this program in a special case in a previous paper 
\cite{HarTym09}; one of the main goals of the current manuscript is to
both generalize and formalize the techniques therein.

Let $G$ be a topological group and suppose that $X$ is a
$G$-space. Throughout this section we assume that the $G$-action
on $X$ is such that the GKM package
holds for $X$ as described in 
Section~\ref{subsec:background-GKM}. Let \(X^G
= \cup_{i \in \mathcal{I}} F_i\) denote the set of (isolated) fixed
points and let $\{x_i\}_{i \in \mathcal{I}}$ denote the choice of
poset-upper-triangular module generators of $E^*_G(X)$. 

We wish to analyze subspaces $Y$ of $X$ using GKM theory on
$X$. For this to be feasible, we need to place certain conditions on
$Y$. We introduce the following terminology. 

\begin{definition}\label{def:GKM compatible}
Let $G$ be a topological group, $X$ a $G$-space, and $E^*_G$ an
equivariant cohomology theory.  
Let \(Y \subseteq X\) be a
subspace of $X$ and \(H \subseteq G\) a topological subgroup of $G$
preserving $Y$.  We call the pair $(Y,H)$ \textbf{GKM-compatible} with
the pair $(X,G)$ with respect to $E^*_G$ if the following conditions hold. 
\begin{enumerate}
\item[(a)] The $H$-fixed set of $Y$ is the intersection of $Y$ with
  $X^G$, i.e. 
\begin{equation}\label{eq:Y-fixed}
Y^H = Y \cap X^G.
\end{equation}
\item[(b)] The $H$-equivariant cohomology $E^*_H(Y)$ is a free
  $E^*_H(\pt)$-module.
\item[(c)] The ring map \(\iota_Y^*: E^*_H(Y) \to E^*_H(Y^H)\) induced
  by the inclusion \(\iota_Y: Y^H \into Y\) is injective. 
\end{enumerate}
\end{definition}

When the groups are clear from context, we may
simply say that $Y$ is GKM-compatible with $X$. Similarly, when there
is no ambiguity we
often neglect to mention the choice of cohomology theory $E^*_G$. 

\begin{remark}
 In certain cases, Conditions (b) and (c) are related. For instance, if
 $G=T$ is a compact torus and $E^*_T = H^*_T(-,\mathbb{F})$ is
 Borel-equivariant cohomology with coefficients in a field $\F$ of
 characteristic zero, then Condition (b) implies Condition
 (c). 
 \end{remark}

 Suppose now that $(Y,H)$ is GKM-compatible with $(X,G)$ with respect
 to $E^*_G$. Since we assume that \(X^G = \cup_{i \in \mathcal{I}}
 F_i\) consists of isolated fixed points, the intersection \(Y \cap
 X^G\) is indexed by some subset $\mathcal{J}$ of $\mathcal{I}$.  In
 this setting, the relationship between 
$E^*_G(X)$ to $E^*_H(Y)$ fits into the algebraic framework
 discussed in Section~\ref{subsection:images of modules}.  Indeed, 
consider first the 
 sequence of ring homomorphisms
\begin{equation}\label{eq:EGX-to-EHY}
\xymatrix{
E^*_G(X) \ar[r] &  E^*_H(X) \ar[r] &  E^*_H(Y) \\
}
\end{equation}
where the first map is the forgetful map associated to the inclusion
of groups \(H \into G\) and the second is induced from the inclusion
of spaces \(Y \into X.\) (For \(X=Y=\pt\), the composition~\eqref{eq:EGX-to-EHY} specializes to the
usual forgetful map \(E^*_G(\pt) \to E^*_H(\pt).\)) 
Moreover, Condition (a) in the definition of GKM-compatibility means that the
map~\eqref{eq:EGX-to-EHY} fits into a
commutative diagram
\begin{equation}\label{eq:main commutative diagram}
\xymatrix @C=1.7in {
E^*_G(X) \ar[d] \ar[r]^{\iota^*} & E^*_G(X^G) \cong \prod_{i \in
  \mathcal{I}} E^*_G(\pt) \ar[d]^\pi \\
E^*_H(X) \ar[r]_{\iota_Y^*} & E^*_H(Y^H = Y \cap X^G) \cong \prod_{i
  \in \mathcal{J}} E^*_H(\pt) 
}
\end{equation}
where the right arrow $\pi$ is $0$ on the components \(i \not \in
\mathcal{J}\) and is the forgetful map \(E^*_G(\pt)\to E^*_H(\pt)\) on
the components \(i \in \mathcal{J}.\) By assumption, the $G$-space $X$
satisfies the GKM conditions, so the restriction map $\iota^*$ is
injective. Finally, Condition (c) of the definition of
GKM-compatibility assures us that $\iota_Y^*$ is an injection, and
Condition (b) ensures that we can find a module basis for
$E^*_H(Y)$. The commutative diagram~\eqref{eq:main commutative
  diagram} is thus an instance of the diagram~\eqref{eq:M-Mprime-diagram} in
Section~\ref{subsection:images of modules}, with $R=E^*_G(\pt),
R'=E^*_H(\pt), M=E^*_G(X)$, and $M' = E^*_H(Y)$.  The forgetful map
\(E^*_G(\pt) \to E^*_H(\pt)\) satisfies the required
property~\eqref{eq:module-twistedring} for \(\gamma: R \to R'\) by
naturality. In contrast to Section~\ref{subsection:images of modules},
we do {\em not} assume here that the map $E^*_G(X) \to E^*_H(Y)$ is
surjective. We discuss this further below.

We first give a rich class of
examples of GKM-compatible subspaces of GKM spaces.

\begin{remark}\label{remark:flags-betti}
  Let \(X= G/B\) be the flag variety of a complex reductive algebraic
  group. As observed in Remark~\ref{remark:GKM conditions}, $X$ is GKM
  with respect to the standard action of the maximal torus $T$ of
  $G$. In Section~\ref{section:petersons} we define a family of
  \textbf{Hessenberg varieties} $Y \subseteq X$ and show that in many
  cases, such as the \textbf{regular nilpotent Hessenberg varieties} in classical Lie type
  and the \textbf{Springer varieties} in Lie type $A$, there is a natural $S^1$ subtorus of $T$ which preserves
  $Y$ and makes $(Y,S^1)$ GKM-compatible with the pair $(X,T)$ with
  respect to Borel-equivariant cohomology with $\mathbb{Q}$ coefficients. 
\end{remark} 

Hessenberg varieties are our primary examples, but the
following question arises naturally. 

\begin{question}
What are other classes of GKM-compatible subspaces of GKM spaces?
\end{question}

We now give the main geometric question of this manuscript.  Let
$\{x_i\}_{i \in \mathcal{I}}$ be a poset-upper-triangular module
basis of $E^*_G(X)$ and let $\overline{x}_i$ denote the image of $x_i$
under the composition \(E^*_G(X) \to E^*_H(Y).\)

 \begin{question}\label{question:main}
   Suppose $(X,G)$ is GKM and suppose $(Y,H)$ is GKM-compatible with
   $(X,G)$. Under what circumstances can we exploit the GKM theory on
   $X$ to explicitly construct a computationally convenient
 module basis for $E^*_H(Y)$ using the images $\{\overline{x}_i\}_{i \in \mathcal{I}}$ in
 $E^*_H(Y)$?
\end{question} 

Ideally we would like the
module basis for $E^*_H(Y)$ to be a linear combination of elements in $\{\overline{x}_i\}_{i \in
  \mathcal{I}}$ or even to be a subset of $\{\overline{x}_i\}_{i \in
  \mathcal{I}}$. For this to
be possible,  we need that
\begin{equation}\label{eq:surjection}
\textup{ the ring homomorphism } E^*_G(X) \to E^*_H(Y) \textup{ is a
  surjection. } 
\end{equation}
Condition~\eqref{eq:surjection} is not included in the definition of
GKM-compatibility because in many applications, we can use
other topological data, together with poset pinball, to deduce surjectivity.
For instance, when \(G=T\) and $E^*_T = H^*_T(-;\F)$, the ring
surjection in Condition~\eqref{eq:surjection} follows from a
successful game of Betti poset pinball; and to play Betti pinball we need prior
knowledge of the target Betti numbers.
We give concrete examples of such arguments in the cases of Peterson varieties in classical Lie type
in Section~\ref{section:petersons}, and Springer varieties in type
$A$ in Section~\ref{sec:springer}. They were also part of our
arguments in previous work \cite{HarTym09}. We discuss this in more
detail in Section~\ref{subsec:borel}.

\begin{remark}
In some situations, 
Condition~\eqref{eq:surjection} may be seen to hold directly without using
Betti-number arguments. 
For instance, suppose the spaces $X$ and $Y$ in the discussion above are complex
  algebraic varieties.  If there is a $G$-invariant affine paving of
  $X$ which restricts to an affine paving of $Y$ (i.e. a
  subset of the affine cells in the $G$-invariant affine paving of $X$
  yields an affine paving of $Y$), then Condition~\eqref{eq:surjection}
  holds for any subgroup $H \subseteq G$. 
\end{remark}

\subsection{Borel-equivariant cohomology $E^*_T = H^*_T$
with field coefficients}\label{subsec:borel}

We now specialize to the case $G=T$ and \(E^*_T =
H^*_T(-;\mathbb{F})\) Borel-equivariant cohomology with coefficients
in a field $\F$ with characteristic zero. Let $X$ denote an ambient $T$-space satisfying the
GKM package and let $(Y,T')$ be a GKM-compatible subspace.
Moreover, let $\{x_i\}$ be a poset-upper-triangular basis
for $H^*_T(X;\F)$ indexed by the set of isolated fixed points
$\mathcal{I} = X^T$ and let $\rho_X: \mathcal{I} \to \N$ be the rank function on the poset
$\mathcal{I}$. 
Borel-equivariant cohomology is a graded theory, so we may speak of
the \emph{degree} of a class $x_i$. 
For most of the discussion we assume that 
\begin{equation}\label{eq:vanish}
\textup{ the ordinary cohomology of the spaces } X \textup{ and } Y \textup{ vanish in odd degrees}.
\end{equation}
In practice, this is not a very restrictive condition, but see
Remark~\ref{remark:mod2} below.

\begin{remark}\label{remark:leray}
  Given the assumptions of this section, and in
  particular the vanishing condition~\eqref{eq:vanish}, it follows that the 
  the Leray-Serre spectral sequence for Borel-equivariant
  cohomology of $Y$ collapses, so $H^*_{T'}(Y;\F)$ is a free
  $H^*_{T'}(\pt;\F)$-module. By the localization theorem in
  Borel-equivariant cohomology (e.g. \cite[Theorem (3.5)]{AtiBot84},
  \cite[Theorem 11.4.4]{GS}) the inclusion \(\iota: Y^{{T'}}
  \into Y\) induces an injection
\[
\iota^*: H^*_{T'}(Y;\F) \into
H^*_{T'}(Y^{T'};\F),
\]
so Conditions (b) and (c) of GKM-compatibility are automatically
satisfied. 
\end{remark}

We need a homogeneity
condition on the classes $x_i$ in the module basis, so we define the following.

\begin{definition}
Let $X$ be as above. 
The poset-upper-triangular basis $\{x_i\}$ of $H^*_T(X;\F)$ is {\em rank-homogeneous} if
each $x_i$ has homogenous degree with respect to the standard
$\mathbb{Z}$-grading on Borel-equivariant cohomology, and if 
\[
\deg(x_i) = 2 \cdot \rho_X(i) \textup{ for all } i \in \mathcal{I}.
\]
\end{definition}
For instance, the bases in 
 Remark~\ref{remark:GKM conditions} for Borel-equivariant cohomology
 are all rank-homogeneous. 
We now tackle Question~\ref{question:main} in this
setting, namely we build a module basis for $H^*_{T'}(Y;\F)$ from
the basis $\{x_i\}_{i \in \mathcal{I}}$ for $H^*_T(X;\F)$. 
As we said at the end of
Section~\ref{subsec:subspace-setup}, we often
know the Betti numbers $b_j := \dim_\F H^{2j}(Y;\F)$ of $Y$ in advance
but do not know that the ring map $H^*_T(X;\F) \to H^*_{T'}(Y;\F)$ is
surjective.  Poset pinball can be useful in this situation: a
successful game of Betti pinball using the target Betti numbers $b_j =
\dim_\F H^{2j}(Y;\F)$
can build a module basis for
$H^*_{T'}(Y;\F)$, from which we may deduce surjectivity.
We explain this in the next two propositions. Recall we denote by
$\overline{x}$ the image in $H^*_{T'}(Y;\F)$ of a class $x$ in
$H^*_T(X;\F)$.

\begin{proposition}\label{prop:assume-betti}
  Let $T$ be a compact torus and $X$ a $T$-space for which the GKM
  package holds. Let $T' \subseteq T$ be a subtorus and suppose
  $(Y,T')$ is GKM-compatible with $(X,T)$. We assume $H^*(Y;\F)$ is
  finite-dimensional.  Suppose $\{x_i\}_{i \in {\mathcal I}}$ is a set
  of rank-homogeneous poset-upper-triangular $H^*_T(\pt,\F)$-module
  generators of $H^*_T(X,\F)$.  Suppose the ordinary cohomology of $Y$
  vanishes in odd degrees and let $b_j := \dim_\F H^{2j}(Y;\F)$ be the
  even Betti numbers of $Y$. 
  Suppose there exists a subset \({\mathcal K}
  \subseteq {\mathcal I}\) such that the images \({\mathcal B}' :=
  \{\overline{x}_k\}_{k \in {\mathcal K}}\) in $H^*_{T'}(Y;\F)$ under
  the ring map $H^*_T(X;\F) \to H^*_{T'}(Y;\F)$
  in~\eqref{eq:EGX-to-EHY}
    \begin{enumerate}
    \item[(A)] are $H^*_{T'}(\pt,\F)$-linearly independent in $H^*_{T'}(Y;\F)$,
      and 
   \item[(B)] for each \(j \in \Z_{ \geq 0}\) there exist precisely
     $b_j$ elements in ${\mathcal B}'$ of homogeneous degree
     $2j$. 
    \end{enumerate}
Then ${\mathcal B}'$ is a $H^*_{T'}(\pt,\F)$-module basis for
$H^*_{T'}(Y)$. 
  \end{proposition}

\begin{proof}
We apply \cite[Proposition A.1]{HarTym09} using \(R = H^*_{T'}(\pt;\F), M =
H^*_{T'}(Y;\F)\), and \(V = H^*(Y;\F)\), where conditions (A) and (B)
above are the hypotheses of  \cite[Proposition
A.1]{HarTym09}.  
\end{proof}

If a subset $\mathcal{K}$ in the above proposition exists, then
the $\{\overline{x}_k\}_{k \in \mathcal{K}}$ form a module basis for
$H^*_{T'}(Y;\F)$, and so the ring map \(H^*_T(X;\F) \to
H^*_{T'}(Y;\F)\) is surjective. In other words, we may deduce
surjectivity from the Betti numbers---though it may be challenging to
find a subset $\mathcal{K}$.  The following proposition shows that
poset pinball can sometimes accomplish this.

\begin{proposition}\label{prop:pinball-betti}
  Suppose we have $T, X, T', Y, \{x_i\}_{i \in \mathcal{I}}, b_j :=
  \dim_\F H^{2j}(Y;\F)$ as in
  Proposition~\ref{prop:assume-betti}. Let \(\mathcal{J}\)
  be the $T'$-fixed points of $Y$ and $\mathcal{I}$ be the $T$-fixed points of $X$. Suppose that {\bf either} of the
  following holds for poset pinball played with ambient poset
  $\mathcal{I}$ and initial subset $\mathcal{J}$: 
 \begin{enumerate}
  \item An instance of upper-triangular Betti pinball with target
    Betti numbers $\mathbf{b} = (b_0,b_1,b_2,\ldots)$ is successful.
  \item An instance of Betti pinball with target Betti numbers
    $\mathbf{b} = (b_0,b_1,b_2,\ldots)$ is successful, and the classes
    corresponding to the rolldown set 
    $\{\overline{x}_{roll(j)}\}_{j \in \mathcal{J}}$ are
    $H^*_{T'}(\pt;\F)$-linearly independent. 
  \end{enumerate} 
Then the classes corresponding to the rolldown set
$\{\overline{x}_{roll(j)}\}_{j \in \mathcal{J}}$ are an
$H^*_{T'}(Y;\F)$-module basis for $H^*_{T'}(Y;\F)$. 
\end{proposition}

\begin{proof} 
If successful, Betti pinball yields the correct
nonzero Betti numbers $(b_0,b_1,b_2,\ldots)$ by construction. Hence in either case, Condition (B) of Proposition~\ref{prop:assume-betti} is satisfied
by the classes corresponding to the rolldowns $roll(j)$ for $j \in
\mathcal{J}$. It remains to verify Condition (A). In the first case, the classes corresponding to the rolldowns are
poset-upper-triangular  by the rules of
upper-triangular pinball; by Proposition~\ref{proposition:independent} they
are linearly independent. In the second, linear independence is
assumed.  The result follows. 
\end{proof} 

\begin{remark}
  Alternatively, if we know in advance that the map $H^*_T(X;\F) \rightarrow
  H^*_T(Y;\F)$ is surjective, then it would be possible to deduce the
  Betti numbers of $Y$ from a successful game of upper-triangular
  pinball together with an argument that the classes corresponding to the
  rolldowns satisfy the conditions of Proposition~\ref{proposition:generate}.
\end{remark}

\begin{remark}\label{remark:billey}
  Betti pinball often yields a set of elements in $H^*_{T'}(Y;\F)$
  which are {\em not} upper-triangular with respect to the original
  partial order $<$ on $\mathcal{I}$. In fact,
  Example~\ref{example:SprN=4,X=211} shows a case in which {\em no}
  successful game of Betti pinball yields a poset-upper-triangular set
  of elements in $H^*_{T'}(Y;\F)$.  However we can frequently find a
  total order $\prec$ compatible with the original 
  partial order, with respect to which
  the classes associated to the rolldown set \emph{are}
  upper-triangular. Applying Proposition~\ref{proposition:independent}
  with respect to the total order $\prec$ we conclude that these
  classes are linearly independent. If Betti pinball is successful in
  this case, Proposition~\ref{prop:pinball-betti} guarantees that the
  set forms a module basis. In our view, this indicates that except in 
  particularly nice settings such as the partial flag varieties $G/P$, the notion
  of poset-upper-triangularity that is geometrically natural may not 
  coincide with the notion that is combinatorially natural (as
  articulated by Billey).
\end{remark}

\begin{remark}\label{remark:mod2}
It is sometimes possible to apply the theory of GKM spaces and
GKM-compatible subspaces without Assumption~\eqref{eq:vanish} on the
ordinary cohomology of $X$ and $Y$. For instance ``mod $2$'' GKM
theory (e.g. \cite{HarHol05} and references therein) can deal with
\emph{real} projective spaces $\R \P^n$ which have non-vanishing cohomology in odd
degree (working with coefficients in $\F = \Z/2\Z$). It would be
possible to rephrase our theorems and the game of poset pinball to
account for these `mod $2$' GKM spaces and Borel-equivariant
cohomology with $\Z/2\Z$ coefficients, but we have chosen to simplify
exposition by assuming~\eqref{eq:vanish}. 
\end{remark}

We close this section by addressing Question~\ref{question:main} via \textbf{matchings} with
respect to different integer
functions on the underlying posets. This is
a complementary approach to pinball that is easier to use in some
contexts. 
Betti pinball does not
determine the degree $\rho_X(roll(j))$ of the
rolldown of the vertex $j \in \mathcal{J}$. However, in some cases, geometric considerations
on the subspace $Y$ naturally give rise to a function $\deg_Y:
\mathcal{J} \to \Z_{\geq 0}$ with the property that
\[b_k = \dim H^{2k}(Y;\F) = | \{j \in \mathcal{J}: \deg_Y(j) = k\}
|.\] For example, as in the setting of
Remark~\ref{remark:flags-betti}, many varieties $Y$ have a paving by
complex affine cells; in this situation, it is natural to define
$\deg_Y(j)$ to be the complex dimension of the affine cell of $Y$
containing the $T'$-fixed point of $Y$ associated to \(j \in
\mathcal{J}.\) Given $\deg_Y$ we could define a new version of poset
pinball in which we require that rolldowns satisfy \(\rho_X(roll(j)) =
\deg_Y(j).\) However, instead of taking
this approach, we construct module bases for $H^*_{T'}(Y;\F)$
directly, using {\em matchings} compatible with rank functions.

\begin{theorem}\label{theorem:matching}
Let $T$ be a compact torus and $X$ a $T$-space for which the GKM
package holds. Let $T' \subseteq T$ be a subtorus and suppose $(Y,T')$
is GKM-compatible with $(X,T)$. We assume 
$H^*(Y;\F)$ is finite-dimensional.
Suppose $\{x_i\}_{i \in
      {\mathcal I}}$ is a set
    of rank-homogeneous poset-upper-triangular $H^*_T(\pt,\F)$-module generators of
    $H^*_T(X,\F)$
    with $min(x_i)=i$ for each $i$.  Suppose the ordinary cohomology of $Y$
    vanishes in odd degrees and let $b_k := \dim_\F H^{2k}(Y;\F)$ be the
  even Betti numbers of $Y$.
    Suppose \(\deg_Y: \mathcal{J} \to \Z_{\geq 0}\) is a function
    with 
\[ | \{j \in \mathcal{J}: \deg_Y(j) = k\} | = b_k.\]
Suppose there is an injection \(f: \mathcal{J} \to \mathcal{I}\) with
\(\deg_Y(j) = \rho_X(f(j))\) and a total ordering $\prec$ 
compatible with the partial order $<$ with respect to which \(\{\overline{x}_{f(j)}\}\)
is upper-triangular. Then
$\{\overline{x}_{f(j)}\}$ is a module basis for $H^*_{T'}(Y;\F)$. 
\end{theorem}

If it exists, we call the map \(f: \mathcal{J} \to \mathcal{I}\) a {\em matching}.

\begin{proof}
  By hypothesis, the set $\{\overline{x}_{f(j)}\}_{j \in \mathcal{J}}$
  is upper-triangular with respect to the total ordering $\prec$
  compatible with $<$.  Proposition~\ref{proposition:independent} then
  implies that $\{\overline{x}_{f(j)}\}_{j \in \mathcal{J}}$ is
  linearly independent. The matching condition implies that for each
  \(k \in \Z_{\geq 0}\) there exist precisely $b_k$ elements in
  $\{\overline{x}_{f(j)}\}_{j \in \mathcal{J}}$ of homogeneous degree
  $2k$ Thus the set $\{\overline{x}_{f(j)}\}_{j \in \mathcal{J}}$
  satisfies the hypotheses of \cite[Proposition A.1]{HarTym09}, and
  the result follows.
\end{proof}

\begin{remark} 
Module bases constructed from a matching may be very
different from bases obtained by poset pinball if the degree
function $\deg_Y$ is not compatible with
the partial order on $\mathcal{I}$. On the other hand, if the degree
function $\deg_Y: \mathcal{J} \to \Z_{\geq 0}$ also satisfies 
\begin{equation}\label{eq:degY-vs-degX}
\deg_Y(j) \leq \rho_X(j) \hsm \mbox{for all} \hsm j \in \mathcal{J},
\end{equation}
namely $2\deg_Y(j)$ is bounded by the cohomology degree of $x_j$ in
$H^*(X)$, then a matching basis could arise as a poset pinball basis.
\end{remark}

\section{Example: Peterson varieties and other regular nilpotent
  Hessenberg varieties}\label{section:petersons}

Regular nilpotent Hessenberg varieties are a family of subvarieties of
the flag varieties $G/B$ which fit into the geometric framework of
Section~\ref{subsec:subspace-setup}, so we can analyze their
equivariant cohomology using poset pinball.  We discuss facts
about regular nilpotent Hessenberg varieties in different Lie types in
Section~\ref{subsec:background-regnilp}. Then in
Section~\ref{subsec:petersons} we explicitly calculate the
Borel-equivariant cohomology of Peterson varieties, a collection of
regular nilpotent Hessenberg varieties. To do this, we find a module
basis, and show the basis may be obtained either via poset pinball (as
in Proposition~\ref{prop:pinball-betti}) or by a matching (as in
Theorem~\ref{theorem:matching}).  We studied the case of Lie type $A$
previously \cite{HarTym09}; the results here generalize that earlier
work to all classical Lie types. In this section we always work in
Borel-equivariant cohomology with coefficients in a field $\F$ of
characteristic zero. 

\subsection{Background on regular nilpotent Hessenberg
  varieties}\label{subsec:background-regnilp}

Let $G$ be a complex reductive linear algebraic group, and let 
$B$ and $T \subseteq B$ denote choices of a Borel subgroup and a
maximal torus of $G$, respectively.  We denote by 
$\mathfrak{g}$ and $\mathfrak{b}$ the associated Lie algebras of $G$
and $B$.
The homogeneous space $G/B$ is a generalized flag variety.  A linear
subspace $H \subseteq \mathfrak{g}$ is called a \textbf{Hessenberg space}
if
 \begin{itemize}
\item $H$ contains the Lie algebra $\mathfrak{b}$, and 
\item $H$ is closed under Lie bracket with $\mathfrak{b}$,
  i.e. $[H,\mathfrak{b}] \subseteq H$.  
\end{itemize}
Let $N \in \g$.
The \textbf{Hessenberg variety} $\mathcal{H}(N,H)$ associated to $N$ and $H$ is the
subvariety of $G/B$ defined by
\begin{equation}\label{eq:def-Hess}
\mathcal{H}(N,H) := \{gB: \textup{Ad} (g^{-1})(N) \in H\} \subseteq
G/B.
\end{equation}
When $N$ is regular (also
called principal) nilpotent, then the Hessenberg variety
$\mathcal{H}(N,H)$ is called a \textbf{regular nilpotent Hessenberg
  variety}.

Let $\Phi$ denote the set of roots
of $\g$ and $\Phi^+ \subset \Phi$ be the set of positive roots corresponding to $\mathfrak{b}$.  Let
$\Delta = \{\alpha_i\}$ denote the set of
simple roots in $\Phi^+$.  If
$\alpha \in \Phi$ is a root, let $\mathfrak{g}_{\alpha}$ be its corresponding
root space.  Fix a basis element $E_{\alpha}$ for each $\mathfrak{g}_{\alpha}$. Let $W$ denote the Weyl group associated to $G$. We use the natural action of the maximal torus $T$ on $G/B$ given by
left multiplication on cosets.  The fixed point set $(G/B)^T$ may be naturally
identified with the Weyl group $W$.

We begin with some useful facts. 

\begin{lemma}\label{lemma:facts}
\begin{enumerate}
\item Any regular nilpotent Lie algebra element $N \in \g$ is $G$-conjugate
  to the regular nilpotent element of the form 
\begin{equation}\label{eq:def N_0}
N_0 := \sum_{\alpha_i
    \in \Delta} E_{\alpha_i}. 
\end{equation}
\item Suppose $H \subseteq \g$ is a Hessenberg space $H \subseteq \g$ and $N_1, N_2 \in \g$ are $G$-conjugate. The corresponding varieties $\mathcal{H}(N_1,H)$ and $\mathcal{H}(N_2,H)$ are isomorphic, with
  explicit isomorphism 
\[
gB \in \mathcal{H}(N_1,H) \leftrightarrow g_2gB \in \mathcal{H}(N_2, H),
\]
where $g_2 \in G$ satisfies $N_1 = \textup{Ad}(g_2^{-1}) (N_2)$.
\item If $N \in \mathfrak{g}$ is a sum of simple root vectors, there exists a circle subgroup $S^1$ of the maximal torus $T$
  such that the restriction of the natural $T$-action on $G/B$ to the $S^1$-subgroup preserves
  $\mathcal{H}(N,H)$. Moreover, the points in
  $\mathcal{H}(N, H)$ that are fixed by this $S^1$-action satisfy 
\begin{equation}\label{eq:regHess-fixedpts}
(\mathcal{H}(N, H))^{S^1} = \mathcal{H}(N, H) \cap (G/B)^T. 
\end{equation}
\end{enumerate}
\end{lemma}

\begin{proof}
  Part (1) is a standard result (see e.g. Collingswood-McGovern
  \cite[Theorem 4.1.6]{ColMcg93}).

  Part (2) is a straightforward consequence of
  Definition~\eqref{eq:def-Hess}. 

To prove (3), we explicitly construct the required subgroup $S^1$. By definition
$\mathfrak{g}_{\alpha}$ is an eigenspace for the action of $\textup{Ad
}T$ with eigenfunction $\alpha: T \rightarrow \mathbb{C}^*$.  This
means that $\textup{Ad }t(x) = \alpha(t)x$ for all $x \in
\mathfrak{g}_{\alpha}$ and that $\alpha$ is a character of $T$, which we think of as an element of $\mathfrak{t}^*$.  (See
also \cite[16.4]{Hum75}.) 
The characters $\alpha_1, \alpha_2, \ldots, \alpha_n$ form a maximal $\mathbb{Z}$-linearly                                
independent set in $\t^*$ by definition of simple roots,
so the map $\phi: T \rightarrow \left(\mathbb{C}^*\right)^n$ given by
$\phi(t) = (\alpha_1(t), \alpha_2(t), \ldots, \alpha_n(t))$ is an isomorphism of linear algebraic groups.

In particular, the preimage of the diagonal subgroup $\{(c,c,\ldots,c) | c \in \mathbb{C}^*\}$
is a rank-one subtorus $S
\cong \mathbb{C}^*$ of $T$ whose elements $t_c$ are parametrized by $c$.
The elements of $S$ also satisfy
\[ (\textup{Ad }t_c) \left(\sum E_{\alpha_i}\right) = \sum c
E_{\alpha_i} = c \left(\sum E_{\alpha_i}\right)\] 
for all $c \in
\mathbb{C}^*$ and any sum of simple root vectors,
since each $E_{\alpha} \in \g_{\alpha}$.  In particular
\( \textup{Ad }(g^{-1}) \left( \textup{Ad }(t_c^{-1})N \right)\) and
\( \textup{Ad }(g^{-1}) N \) differ by a multiple of the nonzero
scalar $c$.  Since $H$ is a vector space we have
\[t_c gB \in \mathcal{H}(N,H) \Longleftrightarrow gB \in  \mathcal{H}(N,H).\]
                                                                                                  
We now confirm that $(G/B)^S = (G/B)^T$.  We saw that the composition of the maps $S$ and $\alpha_i$ send $c \mapsto t_c \mapsto \alpha_i(t_c) = c$ so the composition has degree one for each simple root $\alpha_i$.  Under the natural pairing of
characters and one-parameter subgroups \cite[16.1]{Hum75}, we have 
\begin{equation} \label{eqn: pairing of circle}
\langle S, \alpha_i \rangle = 1 \textup{  for all simple roots  } \alpha_i.
\end{equation}  
 This implies that $S$ is
a regular subgroup \cite[24.4]{Hum75}, from which $(G/B)^S =
(G/B)^T$ follows \cite[Section 24, Exercise 6]{Hum75}.  This in turn
implies Equation \eqref{eq:regHess-fixedpts}.  
Finally, to obtain a real rank-one torus, we may restrict to the unit-length elements
in $\mathbb{C}^*$.
\end{proof}

For the rest of this section, we assume that $N=N_0$, which by
Lemma~\ref{lemma:facts} results in no loss of generality.  We will
also assume that the $S^1$-action on $\mathcal{H}(N_0, H)$ is that
constructed in Lemma~\ref{lemma:facts}. 

Our next goal is to explicitly
describe the $S^1$-fixed points in $\mathcal{H}(N_0, H)$. By
Equation~\eqref{eq:regHess-fixedpts}, this is equivalent to
identifying the $T$-fixed points in $G/B$ that lie in
$\mathcal{H}(N_0, H)$.  The next proposition does this in arbitrary
Lie type.  We need the following notation. 
Given a Hessenberg space $H$, let $\mathcal{M}_H$ denote the set of
roots defined by the condition
\begin{equation}\label{eq:def-MH}
H = \mathfrak{b} \oplus \bigoplus_{\alpha \in \mathcal{M}_H}
\mathfrak{g}_{\alpha}.
\end{equation}
For each \(w \in
W = N(T)/T\), choose a representative \(\tilde{w} \in N(T).\) The coset
$\tilde{w}B$ is independent of the choice of representative
$\tilde{w}$ since \(T \subseteq B\), so we denote it $wB$.  Recall the $T$-fixed
points in $G/B$ are the flags \(\{wB: w \in W\}.\)

\begin{proposition} \label{proposition: Peterson fixed points} 
Let $\g$ be of arbitrary Lie type. Let \(H \subseteq \g\) be a Hessenberg space and let $\mathcal{H}(N_0,
H)$ be the regular nilpotent Hessenberg variety corresponding to $H$
and $N_0$. 
 Then the flag $wB \in (G/B)^T$ is 
  in the regular nilpotent Hessenberg variety $\mathcal{H}(N_0,H)$ if and only if $w^{-1} \Delta
  \subseteq \mathcal{M}_H \cup \Phi^+$.
\end{proposition}

\begin{proof}
  The element $wB$ is in $\mathcal{H}(N_0,H)$ if and only if
  $\textup{Ad}(\tilde{w}^{-1})(N_0) \in H$ for any representative $\tilde{w} \in N(T)$ of $w \in W$. Since $\textup{Ad}(\tilde{w}^{-1})(N_0) =
  \sum_{\alpha_i \in \Delta} E_{w^{-1}\alpha_i}$ 
 we have $wB \in \mathcal{H}(N_0,H)$ exactly if
  $w^{-1} \Delta \subseteq \mathcal{M}_H \cup \Phi^+$.
\end{proof}

We next recall a result which allows us to deduce the 
Betti numbers of 
$\mathcal{H}(N_0, H)$. The original and stronger result, restated below, is that
 certain nilpotent Hessenberg
varieties are \emph{paved by (complex) affines}.

\begin{lemma} \label{lemma: paving dimensions} {\bf (\cite[Theorem 6.1]{Tym06} and \cite[Theorem 4.3]{Tym07})} 
 Assume either that
 \begin{enumerate}
 \item the Lie algebra $\g$ is of classical Lie type and $N=N_0$ is the regular nilpotent element $N_0 = \sum_{\alpha_i \in \Delta} E_{\alpha_i}$ or
 \item the Lie algebra $\g$ is of Lie type $A$ and $N$ is a nilpotent linear operator in Jordan canonical form.
 \end{enumerate}
  Let $H \subseteq \g$ be a
  Hessenberg space and let $\mathcal{H}(N,H)$ denote the
   Hessenberg variety corresponding to $H$ and $N$. 
  Then $\mathcal{H}(N,H)$ has a paving by complex
  affines obtained by intersecting with an appropriate Bruhat
  decomposition $\bigcup \mathcal{C}_w$ of $G/B$.  The homology
  classes corresponding to 
  the subspaces \(\overline{\mathcal{C}_w \cap
  \mathcal{H}(N,H)}\) generate $H_*(\mathcal{H}(N,H))$.  
  
  Moreover in Case (1), the intersection \({\mathcal{C}_w} \cap
  \mathcal{H}(N_0,H)\) is nonempty exactly when $w^{-1} \Delta \subseteq \mathcal{M}_H \cup \Phi^+$.  The
  degree of the homology class corresponding to $w$ is 
    \[2|\{\alpha \in  \Phi^+: w^{-1}(\alpha) \in \mathcal{M}_H\}|.\]  
    In particular, the homology of $\mathcal{H}(N_0, H)$ is
    $\Z$-torsion-free, and nonzero only in even degree.  It follows that the
    $2j^{th}$ Betti number $b_{2j} = \dim_{\F} H^{2j}(\mathcal{H}(N_0,
H);\F)$ is
\begin{equation}\label{eq:betti of Hessenbergs}
b_{2j}= \left|\left\{w \in W: w^{-1} \Delta \subseteq \mathcal{M}_H \cup \Phi^+ \textup{ and } j = |\{\alpha \in
  \Phi^+: w^{-1}(\alpha) \in \mathcal{M}_H\}|  \right\} \right|.
\end{equation}
\end{lemma}

We saw in Remark~\ref{remark:GKM conditions} that the $T$-space $G/B$
satisfies the GKM package of Section~\ref{subsec:background-GKM}
and that the cohomology ring $H^*_T(G/B;\F)$ has a well-known set of
poset-upper-triangular generators with respect to the Bruhat order on
$(G/B)^T \cong W$: the \textbf{equivariant Schubert classes}
$\{\sigma_v\}_{v \in W}$.  Our goal is to construct
computationally convenient module bases for
$H^*_{S^1}(\mathcal{H}(N_0, H); \F)$ using the equivariant Schubert
classes, according to the point of view laid out in
previous sections. For this
we need the following preliminary observation, which we state in more generality than we use here.

\begin{theorem}\label{theorem:Hessies are GKM-compatible}
  Suppose $N$ is the regular nilpotent operator $\sum_{\alpha_i \in
    \Delta} E_{\alpha_i}$ in classical Lie type, or a nilpotent linear
  operator 
  in Jordan canonical form in Lie type $A$.  The pair $(\mathcal{H}(N,
  H), S^1)$ is GKM-compatible with the pair $(G/B, T)$ with respect to
  Borel-equivariant cohomology $H^*(-;\F)$ with coefficients in a
  field $\F$ of characteristic zero.
\end{theorem}

\begin{proof} 
  A matrix in Jordan canonical form in Lie type $A$ is by definition a
  sum of simple root vectors such as the $E_{\alpha_i}$ above. 
  Equation~\eqref{eq:regHess-fixedpts} of Lemma~\ref{lemma:facts} Part
  (3) gives the first condition of GKM
  compatibility, namely Equation~\eqref{eq:Y-fixed} in the case of
  $\mathcal{H}(N, H)$.  The complex paving by affines of
  $\mathcal{H}(N, H)$ described in Lemma~\ref{lemma: paving
    dimensions} implies that the ordinary cohomology of
  $\mathcal{H}(N, H)$ is zero in odd degrees.  The result follows from the argument in
  Remark~\ref{remark:leray}. 
\end{proof}

In order to effectively compute a module basis for
$H^*_{S^1}(\mathcal{H}(N_0, H);\F)$, we need more information 
about the components $\sigma_v(w)$ of the equivariant Schubert
classes in the direct sum
\begin{equation}\label{eq:GB-inclusion}
H^*_T(G/B;\F) \into H^*_T((G/B)^T;\F) \cong \bigoplus_{w \in W}
H^*_T(\pt; \F) \cong \bigoplus_{w \in W} \Sym_{\F}(\t^*),
\end{equation}
where $\Sym_{\F}(\t^*)$ denotes the ring of polynomials with
coefficients in the field $\F$ on the Lie algebra $\t$.  
Billey gave a complete description of the polynomial $\sigma_v(w)$
for any $v, w \in W$ in arbitrary Lie type \cite[Theorem 4]{Bil99}. We
will only need the following consequences of her formula. 

\begin{proposition}\label{proposition: Billey for flags} {\bf
    (Corollaries of Billey's formula \cite[Theorem 4]{Bil99})}  
Let $v, w \in W$. Let $\sigma_v \in H^*_T(G/B;\F)$ be the equivariant Schubert class
corresponding to $v$. Then: 
\begin{itemize}
\item Given a reduced word decomposition of $w$, the component $\sigma_v(w)$ is a sum of terms, with one summand for
  each reduced subword of 
  $w$ that equals $v$. In particular, 
\[
\sigma_v(w) = 0
\]
if $w \not > v$ in Bruhat order. 
\item Suppose $w > v$ in Bruhat order. Each summand in $\sigma_v(w)$ is a monomial in the positive roots
  $\Phi^+$
   with a positive integer coefficient.
\end{itemize}
\end{proposition}

In this context, the
commutative diagram~\eqref{eq:main commutative diagram}
becomes
\begin{equation}\label{eq:comm diagram for Hess}
 \xymatrix{
H^*_T(G/B;\F) \ar @{^{(}->}[r] \ar[d]_{\pi_{G/B}} & H^*_T((G/B)^T;\F) \cong
\bigoplus_{w \in W} \Sym_{\F}(\t^*) \ar[d]^{\pi} \\
H^*_{S^1}(\mathcal{H}(N_0, H);\F) \ar @{^{(}->}[r] &
H^*_{S^1}((\mathcal{H}(N_0, H))^{S^1};\F) \cong \bigoplus_{\stackrel{w \in W \textup{ such that}}{w^{-1} \Delta
  \subseteq \mathcal{M}_H \cup \Phi^+}} \Sym_{\F}(\textup{Lie}(S^1)^*)
}
\end{equation}
where 
the left vertical arrow $\pi_{G/B}$ is
the composition of the natural maps \(H^*_T(G/B;\F) \to
H^*_{S^1}(G/B;\F)\) and \(H^*_{S^1}(G/B;\F) \to
H^*_{S^1}(\mathcal{H}(N_0, H);\F)\) as in
Section~\ref{subsec:subspace-setup}.  For \(v \in W\) let 
\[p_v := \pi_{G/B}(\sigma_v) \in H^*_{S^1}(\mathcal{H}(N_0, H);\F)\] 
denote 
the image of a Schubert class $\sigma_v$ under $\pi_{G/B}$.
Given this
setup, the following proposition---which holds in arbitrary Lie type---is a straightforward consequence of
Proposition~\ref{proposition: Billey for flags}.  

\begin{proposition} \label{proposition: Billey for nilps} 
  Let $H \subseteq \g$ be a
  Hessenberg space and let $\mathcal{H}(N_0,H)$ denote the
  regular nilpotent Hessenberg variety corresponding to $H$ and 
 $N_0$.
    Let \(w \in W\) satisfy \(w^{-1} \Delta
  \subseteq \mathcal{M}_H \cup \Phi^+.\) 
For \(v \in W\),
\begin{itemize}
\item  $p_v(w) = 0$ if $w \not > v$ in Bruhat order, and
\item $p_v(w) \neq 0$ if $w > v$ in Bruhat order.
\end{itemize}
\end{proposition}

\begin{proof}
Since $S^1$ is a torus of rank one, the symmetric algebra
$\Sym_{\F}(\Lie(S^1)^*)$ may be identified with a polynomial ring in one
variable. Denote this variable $t$. 
Given $S^1 \subseteq T$ constructed in Lemma~\ref{lemma:facts} Part (3), consider the
natural map \(H^*_T(\pt;\F)  \to
H^*_{S^1}(\pt;\F)\).  The induced map \( \Sym_{\F}(\t^*) \to \Sym_{\F}(\Lie(S^1)^*)\) sends each simple root
$\alpha_i \in \t^* \subseteq \Sym_{\F}(\t^*)$ to $t \in  \Sym_{\F}(\Lie(S^1)^*)$ by Equation~\eqref{eqn: pairing of circle}.  Moreover, the arrow labeled $\pi$ in Equation~\eqref{eq:comm diagram for Hess} is defined by
restricting the class $(p(w))_{w \in W} \in H^*_T((G/B)^T;\F)$ to the components indexed by 
Weyl group elements with \(w^{-1} \Delta
  \subseteq \mathcal{M}_H \cup \Phi^+\).  In other words $\pi$ sends $(p(w))_{w \in W}$ to \(  (p(w))_{w \in W: w^{-1} \Delta
  \subseteq \mathcal{M}_H \cup \Phi^+} \). 
  Proposition~\ref{proposition: Billey for flags} now implies that $p_v(w)$ is either zero
  or a polynomial in $t$ with positive integer coefficients, so the claim follows.
\end{proof}

\subsection{The $S^1$-equivariant cohomology of the Peterson variety
  in classical Lie types}\label{subsec:petersons}

In this section, we explicitly build module bases for the
$S^1$-equivariant cohomology of \textbf{Peterson varieties} in
classical Lie type. These are special cases of the regular nilpotent
Hessenberg varieties in
Section~\ref{subsec:background-regnilp}, for which the Hessenberg space
is chosen to be
\begin{equation}\label{eq:def H Delta}
H_\Delta := \mathfrak{b} \oplus_{\alpha \in -\Delta}
\mathfrak{g}_{\alpha}.
\end{equation}
In other words \(\mathcal{M}_H = -\Delta.\) 
For notational simplicity, we fix $\mathfrak{g}$ and denote the
corresponding Peterson variety
\[
\mathcal{Y} := \mathcal{H}(N_0, H_\Delta).
\]

The set of $S^1$-fixed points $\mathcal{Y}^{S^1}$ is a key ingredient
in the combinatorial constructions from Section \ref{sec:pinball},
since it corresponds to the initial subset $\mathcal{J} \subseteq
\mathcal{I}$ with $\mathcal{I} = W \cong (G/B)^T$.
Proposition~\ref{proposition: Peterson fixed points} characterizes the
Weyl group elements $w \in W$ whose flags $wB$ are in an arbitrary
regular nilpotent Hessenberg variety; Proposition~\ref{proposition:
  fixed points enumeration} refines this characterization for Peterson
varieties.  We need the following lemma, which we give for
convenience, though it is probably familiar to experts.  We follow the
notation of Section~\ref{subsec:background-regnilp}. Also, given a
subset $J \subseteq \Delta$ we denote by $W_J$ the corresponding Weyl
group, and by $\Phi_J$ (respectively $\Phi_J^+$ or $\Phi_J^-$)
the corresponding root system (respectively positive or negative
roots).

\begin{lemma} \label{lemma: Peterson fixed point description} Let
  $\Phi$ be a finite root system of arbitrary Lie type with Weyl group
  $W$. Let \(\Delta = \{\alpha_i\}\) denote the simple roots in a
  choice of positive roots $\Phi^+$ of $\Phi$. Let \(w \in W\) and
  define
\begin{equation}\label{eq:def J}
J:= \{\alpha_i: w^{-1}(\alpha_i) < 0\}.
\end{equation}
Then 
\begin{equation}\label{eq:Delta condition}
w^{-1} \Delta \subseteq -\Delta
  \cup \Phi^+
\end{equation}
if and only if $w$ is the maximal element of the Weyl
  group $W_J$. 
\end{lemma}

\begin{proof}
   Bourbaki proves that if $w_J$ is the (unique) maximal element of
  $W_J$ then $\ell(w_J) = \Phi^+_J$ and $w_J^{-1}$ sends $\Delta_J$ to
  $-\Delta_J$ \cite[Corollary 3 in VI.1.1.6]{Bou81}. 
  Every element of $\Phi^+_J$ is a linear combination of the simple
  roots $\Delta_J$ with nonnegative coefficients, so
  $w_J^{-1}(\Phi^+_J) = \Phi^-_J$ by linearity.  The length of $w_J$
  is the number of positive roots that $w_J^{-1}$ sends to negative
  roots \cite[Corollary 2 in VI.1.1.6]{Bou81} so $w_J^{-1}(\Phi^+ -
  \Phi^+_J) \subseteq \Phi^+$.  We conclude $w_J^{-1}\Delta \subseteq
  -\Delta \cup \Phi^+$ as desired.

Conversely, suppose $w\in W$ and $w^{-1} \Delta \subseteq -\Delta \cup \Phi^+$.   We first show that
  $w^{-1}(\Phi^+-\Phi_J^+) \subseteq \Phi^+$. 
  Choose $\alpha = \sum_{\alpha_i \in \Delta} c_i
  \alpha_i$ to be an arbitrary positive root with $w^{-1}(\alpha)<0$.
(In particular, each coefficient $c_i$ is non-negative.)  
Since $w^{-1}$ is linear on $\Phi$ we obtain 
\begin{equation}\label{eq:w inverse alpha negative}
w^{-1}(\alpha) =
  \sum_{\alpha_i \in \Delta} c_i
  w^{-1}(\alpha_i) = \sum_{\alpha_k \not \in J} c_k w^{-1}(\alpha_k) + \sum_{\alpha_j \in J}
  c_j w^{-1}(\alpha_j) < 0.
\end{equation}
Now let $\alpha_i \in \Delta$ with $\alpha_i \not \in J$, so $w^{-1}(\alpha_i)
  > 0$.  If the coefficient $c_i$ of $\alpha_i$ is strictly
  positive then
 \begin{equation}\label{eq:inequality}
0 < c_i w^{-1}(\alpha_i) \leq \sum_{\alpha_k \not \in J} c_k w^{-1}(\alpha_k)
  < \sum_{\alpha_j \in J} -c_j w^{-1}(\alpha_j),
\end{equation} 
where the last inequality follows from Equation~\eqref{eq:w inverse alpha
  negative}. Each summand in Equation~\eqref{eq:inequality} is non-negative by definition of $J$
and the fact that the coefficients $c_j, c_k$ are non-negative. 
For each $j \in J$ the root
  $-w^{-1}(\alpha_j)$ is in $\Delta$ by the hypothesis that $w^{-1}\Delta
  \subseteq -\Delta \cup \Phi^+$.  Equation~\eqref{eq:inequality} now implies that $w^{-1}(\alpha_i)$
  is a linear combination of $\{w^{-1}(\alpha_j)\}_{\alpha_j \in J}$.  This contradicts the fact that $w^{-1}(\Delta)$  is a
  linearly independent set of roots.  We conclude that $c_i = 0$ for
  all $i \not \in J$, from which it follows that $w^{-1}(\Phi^+ -
  \Phi_J^+) \subseteq \Phi^+$. 

  To complete the proof, we will show that $w^{-1}w_J = e$. We saw that $w^{-1}(\Phi^+ -
  \Phi_J^+) \subseteq \Phi^+$ and as before we know that
  $w^{-1}(\Phi^+_J) \subseteq \Phi^-$ by linearity. Hence
  $\ell(w) = |\Phi^+_J| = \ell(w_J)$.  
  Recall that $w_J = w_J^{-1}$ \cite[Corollary 3 in
  VI.1.1.6]{Bou81}.  We already saw that $w_J(\Phi^+_J)= -\Phi^+_J$
  and so  
  $w_J(\Phi^+-\Phi^+_J) = \Phi^+-\Phi^+_J$.
  We now show that $w^{-1} w_J (\alpha_i) > 0$ for all
  $i=1,2,\ldots,n$.  If $i \in J$ then since
  \[w^{-1}w_J(\Phi^+_J) = w^{-1}(-\Phi^+_J) \subseteq \Phi^+\] we know
  $w^{-1}w_J(\alpha_i) > 0$.  If $i \not \in J$ then
  $w^{-1}w_J(\alpha_i) = w^{-1}(\alpha)$ for some $\alpha \in \Phi^+ -
  \Phi^+_J$.  The element $w^{-1}$ has length $|\Phi^+_J|$ and
  sends exactly the positive roots $\Phi^+_J$ to negative roots, so
  $w^{-1}(\alpha) > 0$.  Thus $w^{-1}w_J$ sends no simple root to a
  negative root.  It follows that $w^{-1}w_J = e$ \cite[10.2,
  Corollary to Lemma C]{Hum75}, as desired.
\end{proof}

Combining the previous lemma and Proposition~\ref{proposition: Peterson
  fixed points} immediately gives the following. 

\begin{proposition}\label{proposition: fixed points enumeration}
Let $\g$ be of arbitrary Lie type. Fix the
Hessenberg space $H_\Delta := \mathfrak{b} \oplus_{\alpha \in -\Delta}
\mathfrak{g}_{\alpha}$
 and let $\mathcal{Y}$
be the Peterson variety corresponding to $H_\Delta$ 
and $N_0$. Let $\Phi^+ \subseteq
\Phi$ be the positive roots corresponding to $\mathfrak{b}$ in the root system $\Phi$ of $\g$
and let $\Delta \subseteq \Phi^+$ denote the simple roots. 
Then the $S^1$-fixed points $\mathcal{Y}^{S^1}$ of $\mathcal{Y}$ are
in one-to-one correspondence with subsets $J$ of $\Delta$, with
explicit bijection given by 
\begin{equation}\label{eq:bijection}
J \subseteq \Delta \longleftrightarrow \textup{  the maximal element $w_J$ in
 } W_J \longleftrightarrow w_JB \in \mathcal{Y}^{S^1}. 
\end{equation}
\end{proposition}

Now assume $\g$ has classical Lie type. 
In the next result, we combine the concrete characterization of $\mathcal{Y}^{S^1}$ with
the Betti numbers from Lemma~\ref{lemma: paving dimensions}
to construct a rolldown $roll(w)\in W$ for each element $w$ in the initial subset $\mathcal{J} =
\mathcal{Y}^{S^1}$.   We show that the rolldowns can arise from a
successful game of upper-triangular Betti pinball, so the set
$\{p_{roll(w)}\}_{w \in \mathcal{Y}^{S^1}}$ forms an
$H^*_{S^1}(\pt;\F)$-module basis for $H^*_{S^1}(\mathcal{Y};\F)$. 

We denote by $s_i \in W$ the simple reflection corresponding to the simple root $\alpha_i$ in $\Delta$. 

\begin{theorem} \label{theorem: classical basis for Peterson} 
Let $\g$
  be of classical Lie type.   
Fix the
Hessenberg space $H_\Delta := \mathfrak{b} \oplus_{\alpha \in -\Delta}
\mathfrak{g}_{\alpha}$
 and let $\mathcal{Y}$
be the Peterson variety corresponding to $H_\Delta$ 
and $N_0$.
For each subset $J \subseteq \Delta$ let
  $w_J$ be maximal Weyl group element of $W_J$.
  Suppose $J = \{\alpha_{i_1}, \alpha_{i_2}, \ldots, \alpha_{i_s}: i_1
  < i_2 < \ldots < i_s\}$. Let 
\begin{equation}\label{eq:def vJ}
v_J := s_{i_1} s_{i_2} \ldots s_{i_s} \in W. 
\end{equation}
Then the association $w_J \mapsto v_J$ for $J \subseteq \Delta$ is a possible outcome of
  a successful game of upper-triangular Betti pinball, where $v_J =
  roll(w_J)$. 
In particular, the equivariant cohomology classes 
$\{p_{v_J} 
\} \subseteq H^*_{S^1}(\mathcal{Y})$
form a $H^*_{S^1}(\pt;\F)$-module basis for 
  $H^*_{S^1}(\mathcal{Y}; \F)$.
\end{theorem}

\begin{proof}
  We prove the claim by playing upper-triangular Betti pinball. The
  board $\mathcal{I}$ is the Weyl group $W$, identified with the set
  of $T$-fixed points on $G/B$ and equipped with Bruhat order. We use
  the rank function on $\mathcal{I} = W$ given
  by \(\rho(w) = \ell(w)\), the Bruhat length.  The initial subset is
  $\mathcal{J} = \{w_J: J \subseteq \Delta\} \cong \mathcal{Y}^{S^1}$.
  The inclusion \(J' \subseteq J\) implies that \(w_{J'} \in W_J\).
  This in turn implies $w_{J'} \leq w_J$ since $w_J$ is maximal in
  $W_J$.  In other words, Bruhat order induces the partial order on
  the initial subset \(\mathcal{J} \) given by
\begin{equation}\label{eq:partial order} 
w_{J'} < w_J \Leftrightarrow J' \subset J.
\end{equation}
Fix any total order $\prec$ of $\mathcal{J}$ subordinate to this partial order.

Specializing Formula~\eqref{eq:betti of Hessenbergs} in Lemma~\ref{lemma: paving dimensions}
to Peterson varieties, we see that the nonzero Betti numbers of $\mathcal{Y}$
are \(b_{2j} = \binom{|\Delta|}{j}\), namely the number of subsets $J \subseteq \Delta$ with $|J| = j$. 
These $b_{2j}$ are our target Betti numbers.

We now play upper-triangular Betti pinball. We will show that for each
$w_J$ our choice $v_J := roll(w_J)$ satisfies all the rules for basic
pinball, upper-triangular pinball, and Betti pinball.  Indeed, for all
$J$ we have \(v_J <
w_J\) in Bruhat order by construction, so $v_J$ is a possible
 basic pinball rolldown for $w_J$.
Second we prove that at each step of upper-triangular pinball, no wall
on the board prevents $w_J$ from rolling down to $v_J$. It suffices to
show that if $w_{J'} \prec w_J$ then the element $v_J \not < w_{J'}$
in Bruhat order. The reflection $s_{\alpha_i} < v_J$ 
  precisely when \(\alpha_i \in J\) by construction of the element $v_J$.
Hence \(v_J <
  w_{J'}\) if and only if \(J \subseteq J'\), from which it follows
  that if \(w_{J'} \prec w_J\) then \(v_J \not \leq w_{J'}.\) 
Third, we saw above that $\rho(v_J) = \ell(v_J) = |J|$ and there
are precisely $\binom{|\Delta|}{j}$ subsets $J$ with degree $|J|$, so 
the $v_J$ are also rolldowns in Betti pinball. 

Finally, since the equivariant Schubert classes $\{\sigma_w\}_{w \in W}$
are a rank-homogeneous poset-upper-triangular basis with respect to Bruhat
order, we conclude from Proposition~\ref{prop:pinball-betti} 
that the classes $\{p_{v_J}: J \subseteq
\Delta\}$ form a $H^*_{S^1}(\pt;\F)$-module basis for
$H^*_{S^1}(\mathcal{Y};\F)$, as desired. 
\end{proof}

\begin{remark}
A subset of the images $\{p_w\}_{w \in W}$ of the equivariant
Schubert classes generate the ring $H^*_{S^1}(\mathcal{Y};\F)$, so the ring map 
\[
\pi_{G/B}: H^*_T(G/B;\F) \to H^*_{S^1}(\mathcal{Y};\F)
\]
is surjective when $\mathcal{Y}$ is the Peterson variety corresponding
to $G$ of classical Lie type. (This is the map from
Diagram~\eqref{eq:comm diagram for Hess}, which is
Diagram~\eqref{eq:EGX-to-EHY} for the special case of regular nilpotent
Hessenberg varieties.) We note that Carrell and Kaveh have shown that
surjectivity of $\pi_{G/B}$ is equivalent to the statement that
$H^*_{S^1}(\mathcal{Y};\F)$ is generated by the Chern classes of
$B$-equivariant vector bundles \cite{CarKav08}.
\end{remark}

We can also construct the module basis  
  $\{p_{v_J}\}_{J \subseteq \Delta}$ in the above theorem 
from a matching compatible with degrees, as
discussed in Section~\ref{subsec:borel}.  The additional
ingredient which enables this construction is the
geometric data of the dimensions of the affine cells that pave
$\mathcal{Y}$, as recorded in Lemma~\ref{lemma: paving dimensions}.

\begin{theorem}\label{theorem:peterson-matching}
  Let \(\g, \mathcal{Y}, \{w_J: J \subseteq \Delta\}, \{v_J: J
  \subseteq \Delta\}\) be as in Theorem~\ref{theorem: classical basis
    for Peterson}. Then the $H^*_{S^1}(\pt;\F)$-module basis
  $\{p_{v_J}: J \subseteq \Delta\}$ 
of $H^*_{S^1}(\mathcal{Y};\F)$ can be obtained via a matching
compatible with degrees in the sense of
Theorem~\ref{theorem:matching}. 
\end{theorem}

\begin{proof}
  Define a degree function $\deg_{\mathcal{Y}}: \mathcal{J} 
  \to \Z_{\geq 0}$ by
\[
\deg_{\mathcal{Y}}(w_J) = \left( \textup{ complex dimension of the affine
  cell } \overline{\mathcal{C}_{w_J}} \cap \mathcal{Y} \textup{ associated
  to } w_J \textup{ in Lemma~\ref{lemma: paving dimensions} }
\right). 
\]
 Lemmas~\ref{lemma: paving dimensions} and~\ref{lemma: Peterson fixed point description} together show that \(\deg_{\mathcal{Y}}(w_J) =
|J|.\) Take the rank function \(\rho: \mathcal{I} 
\to \Z_{\geq 0}\) to be the usual Bruhat length, namely \(\rho(w) =
\ell(w).\) (Bruhat length of $w$ equals
half the cohomology degree of the Schubert class $\sigma_w$.)  
This means that \(\deg_{\mathcal{Y}}(w_J) =
\rho(v_J).\)  In particular the map \(f: \mathcal{J} 
\to \mathcal{I} 
\) given by \(f(w_J) = v_J\) satisfies
\(\deg_{\mathcal{Y}}(w_J) = \rho(v_J) = |J|.\) The proof of 
Theorem~\ref{theorem: classical basis for Peterson} showed that the
$\{p_{v_J}\}$ are a rank-homogeneous poset-upper-triangular basis with respect to any total ordering compatible with $<$.  
The result
follows from Theorem~\ref{theorem:matching}. 
\end{proof}

In previous work \cite{HarTym09}, the authors constructed an
$H^*_{S^1}(\pt;\F)$-module basis for $H^*_{S^1}(\mathcal{Y};\F)$
without reference to poset pinball, in the case when $\g$ has Lie type
$A$. In fact, the formula for $v_J$ given in Equation~\eqref{eq:def
  vJ} generalizes to arbitrary Lie type the explicit formulas for what
was called $v_{\mathcal A}$ in earlier work \cite[Equation (2.7) and
Definition 4.1]{HarTym09}. We deduce that the basis discussed in
\cite{HarTym09} in fact arises from poset pinball.

Moreover, our previous paper \cite{HarTym09} used the poset pinball basis $\{p_{v_J}\}$
in Lie type $A$ to
explicitly analyze the structure constants of
  $H^*_{S^1}(\mathcal{Y})$ via a kind of Monk's formula in equivariant
  cohomology.  We conclude this section with a question
  for future work. 

\begin{question}
  What is an explicit combinatorial formula
  for the structure
  constants of $H^*_{S^1}(\mathcal{Y})$ with respect to the basis
  $\{p_{v_J}\}$ in each classical Lie type?
\end{question}

\section{Example: Springer varieties in type
    $A$}\label{sec:springer}

In this section we analyze a special class of nilpotent Hessenberg
varieties in Lie type $A$: {\em Springer varieties}, and in
particular the {\em subregular} Springer varieties. 
The flag variety $GL_n(\mathbb{C})/B$ can be
identified with  
\[
\Flags(\C^n) = \{ V_{\bullet}: 0 \subseteq V_1 \subseteq V_2 \subseteq \cdots
\subseteq V_{n-1} \subseteq \mathbb{C}^n \textup{ such that } \dim_{\C}(V_i) = i\}.
\]  
Suppose $N: \mathbb{C}^n \rightarrow
\mathbb{C}^n$ is a nilpotent linear operator and $\b$ is the standard Borel subalgebra of upper-triangular matrices in $\g$. The \textbf{Springer variety $\mathcal{S}_N$ associated to $N$} is the
Hessenberg variety associated to $N$ and the Hessenberg space $H =
\b$, namely
\begin{equation}\label{eq:def Springer} 
\mathcal{S}_N := \mathcal{H}(N, \b).
\end{equation}
In Lie type $A$, this can be expressed as 
\[
\mathcal{S}_N := \{ V_\bullet: NV_i \subseteq V_i \textup{ for all }
1 \leq i \leq n\} \subseteq \Flags(\C^n).
\]

Springer discovered that the symmetric group $S_n$ acts on the ordinary
cohomology $H^*(\mathcal{S}_N; \C)$ for any Springer variety
\cite{Spr76}. 
This representation is graded by the
degree of the cohomology classes.  Springer also showed that the
top-dimensional cohomology group is
an irreducible representation, and that any 
irreducible representation of $S_n$ arises in this way. Indeed, 
the irreducible representation corresponding
to a partition $\lambda = (\lambda_1 \geq \lambda_2 \geq \cdots \geq \lambda_s)$ arises from the top-dimensional cohomology
of the Springer variety $\mathcal{S}_N$ for $N$ with Jordan canonical
form given by Jordan blocks of size $\lambda_1, \lambda_2, \ldots,
\lambda_s$.

In this section, we use poset pinball to construct an explicit module basis for the
$S^1$-equivariant cohomology with complex coefficients of subregular
Springer varieties of Lie type $A$. Moreover, we
construct an $S_n$-representation  on this explicit module basis
and obtain an equivariant Springer
representation.  Goresky and
MacPherson give a related construction for a different
torus action \cite[Section 7]{GorMac10}.

\subsection{An $S^1$-action and the $S^1$-fixed points of Springer
  varieties} 

In this section, we describe an $S^1$-action on arbitrary Springer
varieties of Lie type $A$ and make some initial observations on their
fundamental properties. For instance, we will see that the fixed
points $\mathcal{S}_N^{S^1}$ may be identified with the set of
permutations whose descents are in positions given by the partition of
$n$ determined by the Jordan canonical form of $N$.  By contrast, Carrell obtains a similar result in general Lie type using a different torus action \cite{Car86}.  In Section~\ref{subsec:subregular}, we will
specialize to a particular class of nilpotent operators called the
\emph{subregular} operators.

Lemma~\ref{lemma:facts} shows that for any $g \in GL_n(\mathbb{C})$ the Springer
variety $\mathcal{S}_N$ is homeomorphic to $\mathcal{S}_{g^{-1}Ng}$.
We will assume without loss of generality
that $N$ is in Jordan canonical form, with Jordan blocks weakly decreasing in size.  We
denote by $\lambda_N$ both the partition of $n$ and the Young diagram
corresponding to this decomposition of $N$ into Jordan blocks.

In Lemma~\ref{lemma:facts} Part (3) we define a circle subgroup of the
standard maximal torus $T^n$ of diagonal matrices in $U(n,\C)$.  It
can be described very explicitly in this setting as 
\begin{equation}\label{eq:def S^1 for Springer}
S^1 := \left\{ \left. \begin{bmatrix} t^n & 0 & \cdots & 0 \\ 0 & t^{n-1} &  &
      0 \\ 0 & 0 & \ddots & 0 \\ 0 & 0 &  & t \end{bmatrix}
  \; \right\rvert \;  t \in \C, \; \|t\| = 1 \right\}  \subseteq T^n
\subseteq U(n,\C).
\end{equation}
The maximal torus $T^n$ acts canonically on $GL(n,\C)/B \cong
\Flags(\C^n)$ so $S^1 \subseteq T^n$ also acts naturally.  
In this case Lemma~\ref{lemma:facts} states
\[
\Flags(\C^n)^{S^1} = \Flags(\C^n)^{T^n}.
\]

We now show that the subgroup $S^1$ in Equation~\eqref{eq:def S^1 for Springer} preserves the
Springer variety $\mathcal{S}_N$. 

\begin{lemma}
Let $N$ be a nilpotent operator in Jordan canonical
form. Then the subgroup $S^1$ in Equation~\eqref{eq:def S^1 for Springer}
preserves the Springer 
  variety $\mathcal{S}_N \subseteq \Flags(\C^n)$.
\end{lemma}

\begin{proof}
Suppose $V_{\bullet} \in \mathcal{S}_N$ and let $\gamma(t)$ denote the diagonal matrix with entries $t^n, t^{n-1}, t^{n-2}, \ldots, t$ along the diagonal, as in Equation~\eqref{eq:def S^1 for Springer}.  Observe that 
\[N \left( \gamma(t) V_i \right) \subseteq \gamma(t)V_i\]
if and only if 
\[ \left( \gamma(t)^{-1}N\gamma(t) \right)V_i \subseteq V_i.\]
A simple calculation shows that
\[ \gamma(t)^{-1}N\gamma(t) = tN.\]
Moreover $(tN)V_i \subseteq V_i$ if and only if $NV_i \subseteq V_i$ since  $V_i$ is linear.  We conclude that $tV_{\bullet}$ is in $\mathcal{S}_N$ as desired.
\end{proof}

The following proposition is a summary of results in the literature,
phrased in our language. 

\begin{proposition}\label{prop:GKM compatibility for Springer}
  For each nilpotent operator $N: \mathbb{C}^n \rightarrow
  \mathbb{C}^n$ the Springer variety $\mathcal{S}_N$ has no
  odd-dimensional cohomology.  If $N$ is in Jordan canonical form and
  $S^1$ is as in Equation~\eqref{eq:def S^1 for Springer} then the pair
  $(\mathcal{S}_N,S^1)$ is GKM-compatible with $(GL_n(\mathbb{C})/B,
  T)$ with respect to Borel-equivariant cohomology $H^*(-; \F)$.
\end{proposition}

\begin{proof}
  Spaltenstein proved that the ordinary cohomology of Springer
  varieties is zero in odd degrees 
  \cite{Spa76}.
  The result follows from the argument in Remark~\ref{remark:leray}.
\end{proof}

We now compute the fixed points of the Springer variety
$\mathcal{S}_N$ with respect to this $S^1$-action.  
Given a partition $\lambda = (\lambda_1 \geq
\lambda_2 \geq \cdots \geq \lambda_s)$ of $n$ the permutation  \(w \in S_n\) has {\em descents in positions given by $\lambda$} if
\[
w(i) > w(i+1) \Rightarrow i \in \{\lambda_1, \lambda_1+\lambda_2,\lambda_1+\lambda_2+\lambda_3,
\cdots,   \lambda_1 + \lambda_2+\lambda_3 + \cdots + \lambda_s \}. 
\]
For example,  the permutation \(w
= (24581736)\) has descents in the positions given by $\lambda = (4,2,2)$.

\begin{theorem}\label{theorem:fixed points of Springers}
  Let $N: \mathbb{C}^n \rightarrow
  \mathbb{C}^n$ be a nilpotent operator in Jordan canonical
  form whose Jordan blocks weakly decrease in size. Let $\lambda_N$ 
  be the corresponding partition of $n$. The $S^1$-fixed points
  of $\mathcal{S}_N$ are in bijection with the 
  set 
\[
\{ w \in S_n: w^{-1} \textup{ has descents in the positions given by } \lambda_N \}.\]
The bijection sends the permutation $w$ to the fixed point  $wB $, where
$w$ also denotes the permutation matrix whose $i^{th}$ column is the standard basis vector $e_{w(i)}$ for all $i$.
\end{theorem}

\begin{proof}

Since $\Flags(\C^n)^{T^n} = \Flags(\C^n)^{S^1}$ it suffices to find the
intersection 
\[
\Flags(\C^n)^{T^n} \cap \mathcal{S}_N.
\]
The $T^n$-fixed points of
$\Flags(\C^n)$ consist precisely of the permutation flags
\(\{wB: w \in S_n\}\) where $w$ is a permutation matrix in $GL(n,\C)$ whose $i^{th}$ column has the standard basis vector $e_{w(i)}$. The definition of Springer varieties
in Equation~\eqref{eq:def Springer} says that $wB$ is in $\mathcal{S}_N$ exactly when ${w}^{-1}N{w}$ is  upper-triangular.  Let $E_{i,j}$ denote the $n \times n$
  matrix with $1$ in the $(i,j)$-th entry and $0$ in all other
  entries.  The matrix $N$ is in Jordan canonical form, so $N = \sum_{i \not \in A} E_{i,i+1}$  where
  \[A = \{ (\lambda_N)_1, (\lambda_N)_1+(\lambda_N)_2,  \cdots, (\lambda_N)_1 +(\lambda_N)_2+(\lambda_N)_3+\cdots+(\lambda_N)_s \}.\]
(In other words, the sum is over pairs $i,i+1$ in the same part of the partition $\lambda_N$.)  
This means ${w}B$ is in $\mathcal{S}_N$ 
  if and only if
 \[{w}^{-1} \sum_{i \not \in A} E_{i,i+1} {w} = \sum_{i \not \in A} E_{w^{-1}(i),w^{-1}(i+1)} \in \b,\] 
or equivalently
\[
i \not \in A \Rightarrow w^{-1}(i) <  w^{-1}(i+1).
\]
As desired, this implies  
  $w^{-1}$ has descents in the positions given by the
  partition $\lambda_N$. 
\end{proof}

\begin{example}
Let \(n=4\) and take $N$ to be the
matrix with $2$ Jordan blocks each of size $2$, so $N = E_{12} + E_{34}$. By
Theorem~\ref{theorem:fixed points of Springers}, the $S^1$-fixed points of
$\mathcal{S}_N$ are the inverses of the following permutations,
written in one-line notation: 
\[
1234, 1324, 1423, 2314, 2413, 3412.
\] 
Informally, the matrix $w^{-1}Nw$ is the sum of $E_{w^{-1}(i), w^{-1}(i+1)}$ over $i$ such that 
 $i,i+1$ are in the same part of the partition $\lambda_N$.
For example, the matrix corresponding to the fixed point
$w^{-1}=2314$ is $E_{23} + E_{14}$ while the matrix corresponding to
$w^{-1}=3412$ is $E_{34}+E_{12}$.
\end{example}

\subsection{The subregular Springer representation} A nilpotent linear operator \(N: \C^n \to \C^n\) is
\textbf{subregular} if the partition associated to its Jordan canonical form is $(n-1,1)$, namely it has one Jordan block of size $n-1$ and
one of size $1$. If $N$ is subregular, the Springer variety
$\mathcal{S}_N$ is called a \textbf{subregular Springer variety} and
Springer's representation
on  $H^*(\mathcal{S}_N;\C)$ is the \textbf{subregular Springer
  representation}.  In this section, we collect results about the subregular Springer variety and representation.
  
Theorem~\ref{theorem:fixed points of Springers} implies
that the $S^1$-fixed points of the subregular Springer variety are in
bijective correspondence with permutations whose descents are given by
the partition $(n-1,1)$. The one-line notation of such a
permutation $w$ increases in the first $n-1$
entries. Since the last entry can be any integer between $1$ and $n$,
the fixed points in one-line notation are precisely 
\[
\mathcal{S}_N^{S^1} = \{w_i := 1 \hsm 2 \hsm 3 \cdots \hsm i-1 \hsm
\hat{i} \hsm  i+1 \cdots n \hsm i \textup{ for each $i$ with } 1
\leq i \leq n\},
\]
where $\hat{i}$ indicates that the integer $i$ is skipped. Note that
$w_n$ is the identity element in $S_n$.

We describe Garsia-Procesi's construction of the subregular Springer
representation, using a classical description of irreducible
representations of $S_n$. A filling of the Young diagram $(n-1,1)$
with the numbers $1,2,\ldots,n$ without repetition is {\em row-strict}
if each row increases left-to-right.  (Hence a row-strict subregular filling
is either a Young tableau or has $1,2,3,\ldots,n-1$ in the top row and
$1$ in the bottom row.)

Let 
$M^{(n-1,1)}$ denote the complex vector space whose basis is the set of
row-strict fillings
with shape $(n-1,1)$. Define an $S_n$-action on $M^{(n-1,1)}$ as
follows.  Given $w \in S_n$ and a row-strict filling $T$ of shape
$(n-1,1)$, define the filling $w(T)$ by:
\begin{itemize}
\item For each $i=1,2,3,\ldots,n$, place $w(i)$ in the box where $T$
  had entry $i$.   
\item Reorder each row so it increases left-to-right.
\end{itemize}
By construction $w(T)$ is row-strict of shape $(n-1,1)$. This is a
well-defined action of $S_n$ on the set of row-strict fillings of
shape $(n-1,1)$, and extends by $\C$-linearity to a representation of
$S_n$ on $M^{(n-1,1)}$.  (These claims generalize to arbitrary
partitions $\lambda$ of $n$, see e.g. \cite[Section
7.2]{Ful97}.)

Combining several results (see Garsia-Procesi's summary \cite[page
84]{GarPro92} and e.g. Fulton's text for background \cite[Section
7.2]{Ful97}) yields the following.

\begin{proposition}\label{prop:Garsia-Procesi} {\bf (Garsia-Procesi)} 
  The subregular Springer representation on $H^*(\mathcal{S}_{(n-1,1)};\C)$ 
is isomorphic, as an ungraded
  representation, to the $S_n$-representation
  $M^{(n-1,1)}$ defined above.  Moreover, suppose $M^{(n-1,1)}$ is given
  the grading inherited from
  the cohomology ring $H^*(\mathcal{S}_{(n-1,1)};\C)$.  Then the set of $n-1$ vectors 
\[
v_j := \begin{array}{|c|c|c|c|} \cline{1-4} 1 & 2 & \cdots & n \\
  \cline{1-4} j & \multicolumn{3}{c}{} \\ \cline{1-1} \end{array}
- \begin{array}{|c|c|c|c|} \cline{1-4} 2 & 3 & \cdots & n \\
  \cline{1-4} 1 & \multicolumn{3}{c}{} \\ \cline{1-1} \end{array}
\]
form a $\C$-basis
  $\{v_2,v_3,\ldots,v_n\}$ for the top-degree graded piece of
  $M^{(n-1,1)}$. 
A basis for the zero-degree graded piece is given by the vector
\[
v_0 := \sum_{j=1}^n \begin{array}{|c|c|c|c|} \cline{1-4} 1 & 2 & \cdots
  & n \\ \cline{1-4} j & \multicolumn{3}{c}{} \\
  \cline{1-1} \end{array}. 
\]
\end{proposition}

We can explicitly compute the character of the subregular Springer representation from this description. The representation
preserves degrees, so we analyze the characters on each degree separately.  Denote the character of the Springer representation on $H^{2i}(\mathcal{S}_N; \C)$ by $\chi^i: S_n \rightarrow \mathbb{Z}$.  For instance, if $w \in S_n$ the integer $\chi^1(w)$ is the trace of the linear operator on $H^2(\mathcal{S}_N; \C)$ corresponding to $w$.

\begin{corollary}\label{corollary:trace}
Let $S_n$ act on $H^*(\mathcal{S}_N;\C) \cong M^{(n-1,1)}$ via the
subregular Springer representation.  Then
\[ 
\chi^1(w)= \# \{\textup{ fixed points of } w\} -1 = \#\{ j \in \{1,2,\ldots,n\}: w(j)= j \} - 1
\]
for all \(w \in S_n\).  Also for all \(w \in S_n\)
\[
\chi^0(w)= 1.
\]
In particular the subregular Springer representation in the
zero-degree piece is the trivial representation.
\end{corollary}

\begin{proof}
  The first part of the claim is an exercise from Sagan \cite[Exercise 2.12.4]{Sag91}, and a nice exercise for the
  reader, given the explicit basis for $H^2(\mathcal{S}_N;\C)$
  described in Proposition~\ref{prop:Garsia-Procesi}.

The second part can be seen by inspection.  By definition $w \cdot v_0
= v_0$ for all $w \in S_n$.  This means $\chi^0(w)=1$ for all $w$, so
the zero-degree piece is the trivial representation, as desired.
\end{proof}

\subsection{Lifting the Springer action to the $S^1$-equivariant cohomology of
  subregular Springer varieties}\label{subsec:subregular}

We now lift the Springer representation on the ordinary cohomology
$H^*(\mathcal{S}_N; \C)$ to an action of $S_n$ on the $S^1$-equivariant cohomology
$H^*_{S^1}(\mathcal{S}_N; \C)$ in the case when $N$ is subregular. To define 
the lift, we first build a
convenient $H^*_{S^1}(\pt;\C)$-module basis of
$H^*_{S^1}(\mathcal{S}_N;\C)$ using upper-triangular Betti pinball.

In this context, the
commutative diagram~\eqref{eq:main commutative diagram} becomes: 
\begin{equation} \label{eqn:Springer diagram}
\xymatrix{
H^*_T(GL(n,\C)/B;\C) \ar @{^{(}->}[r] \ar[d] & H^*_T((GL(n,\C)/B)^T;\C) \cong
\prod_{w \in S_n} H^*_T(\pt;\C) \cong \prod_{w \in S_n} \Sym(\t^*) \ar[d] \\
H^*_{S^1}(\mathcal{S}_N;\C) \ar @{^{(}->}[r] & H^*_{S^1}(\mathcal{S}_N^{S^1};\C)
\cong \prod_{w \in \mathcal{S}_N^{S^1}} H^*_{S^1}(\pt;\C) \cong \prod_{w \in \mathcal{S}_N^{S^1}} \Sym(\Lie(S^1)^*)
}
\end{equation}
where the left vertical arrow is the composition
\[H^*_T(GL(n,\C)/B;\C) \longrightarrow H^*_{S^1}(GL(n,\C)/B;\C) \longrightarrow H^*_{S^1}(\mathcal{S}_N;\C)\]
and the right vertical arrow is zero on each component corresponding to $w \not \in  \mathcal{S}_N^{S^1}$.
As before, we denote the equivariant Schubert class corresponding to
$w$ in $H^*_T(GL(n,\C)/B;\C)$ by $\sigma_w$ for each $w \in S_n$.  

For each $i=1,2,\ldots,n-1$, let $s_i$ be the permutation on
$\{1,2,\ldots,n\}$ that exchanges $i$ and $i+1$ and leaves the other
numbers fixed. (In general Lie type $s_i$ is the reflection
$s_{\alpha_i}$ for a choice of simple roots.)
The proof of the next theorem is similar to that of Theorem~\ref{theorem:
  classical basis for Peterson}. 

\begin{theorem} \label{theorem: pinball for subregular} 
Let $N$ be a subregular nilpotent linear operator $N: \C^n \to \C^n$ and let
$\mathcal{S}_N$ denote its associated subregular Springer variety.  
Then the association 
\[
w_i^{-1} \mapsto v_i := \begin{cases} e \quad \textup{ if } i=n \\ s_i \quad
  \textup{ if } 1 \leq i \leq n-1 \end{cases}
\]
is an outcome of
  a successful game of upper-triangular Betti pinball, where $v_i =
  roll(w_i^{-1})$. 
In particular, the classes 
$\{p_{v_i} \}_{i=1}^n = \{p_e, p_{s_1}, p_{s_2}, \ldots, p_{s_{n-1}}\}$
form an $H^*_{S^1}(\pt;\C)$-module basis for 
  $H^*_{S^1}(\mathcal{S}_N;\C)$.
\end{theorem}

\begin{proof}
Recall that $w_n$ is the identity $e \in S_n$. For each \(i=1,2,\ldots,n-1\) note that
\begin{equation}\label{eq:reduced word decomposition of w_i}
w_i = s_{n-1} s_{n-2} \cdots s_{i+1}s_i
\end{equation}
is a reduced word decomposition of $w_i$. 
Thus for each $i=1,2,\ldots, n-1$, we have $w_i > s_{i}$ and $w_i \not
> s_j$ for any $j < i$. Moreover \(w_n^{-1} < w_{n-1}^{-1} < \cdots <
w_1^{-1}\) so Bruhat order totally orders the $S^1$-fixed points.
In particular the choice of 
total order required to play pinball is uniquely determined in this case.

  Now we play upper-triangular Betti pinball with board $S^n = \left( GL(n,\C)/B \right)^{S^1}$ and initial subset $\left(\mathcal{S}_N\right)^{S^1}$. 
We saw that $v_i  < w_i^{-1}$ for all $i$. In upper-triangular pinball, 
walls are never placed between $w_i^{-1}$ and $v_i$ 
because if $w_j < w_i^{-1}$ then $j>i$ and so $v_i \not < w_j^{-1}$. 
Finally \(\ell(v_n) = 0\) and \(\ell(v_i) = 1\) for \(i=1,2,\ldots,n-1\).  Comparing with 
the Betti numbers of $\mathcal{S}_N$ in Proposition~\ref{prop:Garsia-Procesi}, we conclude
 that the $\{v_i\}_{i=1}^n$  are a 
successful outcome of upper-triangular Betti pinball. 
Applying Proposition~\ref{prop:pinball-betti}  gives the
claim.
\end{proof}

With a basis for $H^*_{S^1}(\mathcal{S}_N;\C)$ in hand, we may now
discuss our construction of an $S_n$-representation, for which we
depend on a previous construction by Kostant and Kumar of a
$W$-representation on $H^*_T(G/B;\C)$ when $G$ is any Kac-Moody group
\cite{KosKum86}. We work in Lie type $A$, for which $W
=S_n$ and $G/B \cong \Flags(\C^n)$. 

In the case of $S_n$ acting on $H^*_T(GL(n,\C)/B;\C)$, Kostant and Kumar's action is defined as
follows. As before, denote the $u$-th coordinate of a class $\sigma \in
H^*_T(GL(n,\C)/B;\C)$ by $\sigma(u)$. Let \(w,u \in W\), and let $\sigma$ be any element of
$H^*_T(GL(n,\C)/B;\C)$. The element $w \cdot \sigma$ is defined by the
equation
\begin{equation}\label{eq:def K-K action}
(w \cdot \sigma)(u) := \sigma(uw).
\end{equation}
This Kostant-Kumar action is defined componentwise, so it commutes with
the natural maps induced on $H^*_{T}(GL(n,\C)/B;\C)$ and
$H^*_{T}((GL(n,\C)/B)^T;\C)$ by $S^1 \hookrightarrow T$.
This implies that Kostant-Kumar's  action descends to an action on
$H^*_{S^1}(GL(n,\C)/B;\C)$. 
(We warn the reader that not all $S_n$-actions on 
$H^*_{T}(G/B;\C)$ descend to $H^*_{S^1}(G/B;\C)$; 
Tymoczko analyzes another natural $S_n$-action
that does \emph{not} \cite{Tym08}.)

The main result of this section is that Kostant-Kumar's  action gives
rise to an $S_n$-action on $H^*_{S^1}(\mathcal{S}_N;\C)$ when
$\mathcal{S}_N$ is a subregular nilpotent Springer variety, and that
this action lifts the Springer representation to $H^*_{S^1}(\mathcal{S}_N;\C)$.
The first step is to show that Kostant-Kumar's  action on
$H^*_{S^1}(GL(n,\C)/B;\C)$ preserves the $H^*_{S^1}(\pt;\C)$-submodule
spanned by the elements corresponding to the rolldowns from
Theorem~\ref{theorem: pinball for subregular}.

The following proposition gives a special case of a more general
formula of Kostant-Kumar \cite[Proposition 4.24.g]{KosKum86}.  
The interested reader may also prove it using the following special cases of Billey's formula for the classes
$\sigma_{s_j}$: 
\[\sigma_{s_j}(w) = \alpha_j \textup{  for  } w \in \{s_j\} \cup \{ s_j s_i: i \neq j, i=1,\ldots, n-1\} \cup \{s_i s_j: i \neq j \pm 1, i \neq j, i = 1,\ldots, n-1\}\] 
and 
\[\sigma_{s_j}(w) = \alpha_i+\alpha_j \textup{  for  } w = s_{j-1}s_j
\textup{  or  } s_{j+1}s_j.\]

\begin{proposition}\label{prop:KK formula}
{\bf (Kostant-Kumar)} For each $i,j$ with $1 \leq i,j \leq n-1$ we have
\begin{itemize} 
\item if $i \neq j$ then
\[s_i \cdot \sigma_{s_j} = \sigma_{s_j},\]
\item if $i=j$ then
\[
s_j \cdot \sigma_{s_j} = \begin{cases} 
\alpha_j \sigma_e - \sigma_{s_j} + \sigma_{s_{j+1}} \quad \quad \quad \quad \quad \textup{if } j=1,\\
\alpha_j \sigma_e - \sigma_{s_j} + \sigma_{s_{j-1}} +
\sigma_{s_{j+1}} \quad \hspace{0.5em} \textup{if } j=2,3,\ldots, n-2, \\
\alpha_j \sigma_e - \sigma_{s_j} + \sigma_{s_{j-1}} \quad \quad \quad \quad \quad  \textup{if } j=n-1,
\\
\end{cases}
\]
\item and for all $w \in S_n$ 
\[w \sigma_e = \sigma_e.\]
\end{itemize} 
\end{proposition}

Our choice of $S^1$ induces the linear projection \(\t^* \to
\Lie(S^1)^*\) which sends the simple roots $\alpha_i$ to $t$, where
$t$ denotes the polynomial variable in $\Sym(\Lie(S^1)^*)$.  The
following corollary is immediate from this observation together with
the formulae in Proposition~\ref{prop:KK formula}.

  \begin{corollary}\label{corollary:KK on flags}
  The Kostant-Kumar action 
  of $S_n$ on $H^*_{S^1}(GL_n(\mathbb{C})/B;\C)$  preserves the $H^*_{S^1}(\pt;\C)$-submodule
  that is spanned by the images of the classes $\{\sigma_e, \sigma_{s_1}, \sigma_{s_2}, \ldots,
  \sigma_{s_{n-1}}\}$. 
  \end{corollary}
  
  \begin{proof}
    By definition, Kostant-Kumar's  action of each $w\in S_n$ is an
    $H^*_T(\pt;\C)$-module homomorphism, in the sense that if $f \in
    \Sym(\t^*)$ and $\sigma \in H^*_T(GL_n(\C)/B; \C)$ then $w \cdot
    (f\sigma) = (f) (w \cdot \sigma)$.    
    Proposition~\ref{prop:KK formula}
    thus implies that the $H^*_T(\pt;\C)$-span of $\{\sigma_e,
    \sigma_{s_1}, \sigma_{s_2}, \ldots, \sigma_{s_{n-1}}\}$ is an
    $S_n$-subrepresentation of $H^*_T(GL_n(\C)/B; \C)$.  The ring
    homomorphism $H^*_T(\pt;\C) \rightarrow H^*_{S^1}(\pt;\C)$ is
    a surjection and the additive homomorphism \\ $H^*_T(GL_n(\C)/B; \C) \rightarrow
    H^*_{S^1}(GL_n(\C)/B; \C)$ 
    respects multiplication in the sense of
    Equation~\eqref{eq:module-twistedring}.  Hence the images of the
    classes $\{\sigma_e, \sigma_{s_1}, \sigma_{s_2}, \ldots,
    \sigma_{s_{n-1}}\}$ span an $H^*_{S^1}(\pt; \C)[S_n]$-submodule of
    $H^*_{S^1}(GL_n(\C)/B; \C)$.
 \end{proof}

As a consequence, we deduce that the formulae given in the following
corollary give well-defined actions on the ordinary and 
equivariant cohomology rings of the subregular Springer
varieties. For
each $w \in S_n$ we denote by $p_w$ the image in
$H^*_{S^1}(\mathcal{S}_N;\C)$ of $\sigma_w$ under the
left vertical arrow of~\eqref{eqn:Springer diagram}.

\begin{corollary}\label{corollary: KK on Springers} 
  Let $N$ be a subregular nilpotent linear operator $N: \C^n \to \C^n$
  and let $\mathcal{S}_N$ denote its associated subregular Springer
  variety.  Kostant-Kumar's $S_n$-action on $H^*_{T}(G/B;\C)$,
  described in Proposition~\ref{prop:KK formula}, naturally induces an
  $S_n$-representation on $H^*_{S^1}(\mathcal{S}_N;\C)$ as
  follows. For each $i,j$ with \(1 \leq i, j \leq n-1\) define: 
\begin{itemize} 
\item if $i \neq j$ then
\[
s_i \cdot p_{s_j} = p_{s_j},
\]
\item if $i=j$ then
\[
s_j \cdot p_{s_j} = \begin{cases} 
t p_e - p_{s_j} + p_{s_{j+1}} \quad \quad \quad \quad \quad  \textup{if } j=1, \\ 
tp_e - p_{s_j} + p_{s_{j-1}} + p_{s_{j+1}} \quad \hspace{0.5em} \textup{if } j=2,3,\ldots, n-2, \\
t p_e - p_{s_j} + p_{s_{j-1}}  \quad \quad \quad \quad \quad  \textup{if } j=n-1, \\
\end{cases}
\]
\item for all \(w \in S_n\)
\[
w \cdot p_e = p_e.
\]
\end{itemize}
This is a well-defined $S_n$-action on
$H^*_{S^1}(\mathcal{S}_N;\C)$. 
Morever, this action 
induces a well-defined $S_n$-representation on the ordinary cohomology
\[H^*(\mathcal{S}_N;\C) \cong H^*_{S^1}(\mathcal{S}_N;\C)/(t)H^*_{S^1}(\mathcal{S}_N;\C)\]
by setting $t=0$ in the previous formulae.
\end{corollary}

\begin{proof}
From Theorem~\ref{theorem: pinball for subregular} we know that
$H^*_{S^1}(\mathcal{S}_N;\C)$ is a free $H^*_{S^1}(\pt;\C)$-module with
basis $\{p_e, p_{s_1}, \ldots, p_{s_{n-1}}\}$.  In addition
$H^*_{T}(GL(n,\C)/B;\C)$ is a free
$H^*_{T}(\pt;\C)$-module with module basis given by the equivariant Schubert classes $\{\sigma_w\}_{w \in S_n}$. 
By Proposition~\ref{prop:KK formula},  Kostant-Kumar's 
  action preserves
  the $H^*_{T}(\pt;\C)$-submodule of $H^*_{T}(GL(n,\C)/B;\C)$ generated
  by the degree-$0$ and degree-$2$ classes \(\{\sigma_e, \sigma_{s_1},
  \ldots, \sigma_{s_{n-1}}\}\).  By definition of the classes $\{p_w\}$, this submodule maps 
  isomorphically onto $H^*_{S^1}(\mathcal{S}_N;\C)$ under the
 natural map 
 \[H^*_{T}(GL(n,\C)/B;\C) \to H^*_{S^1}(\mathcal{S}_N;\C).\] 
 The action of $S_n$ on $H^*_{S^1}(\mathcal{S}_N;\C)$ is defined via this
  isomorphism. The explicit formulas to be proven follow immediately from
  Proposition~\ref{prop:KK formula} and the definition of the classes $p_w$.

  By Proposition~\ref{prop:GKM compatibility for Springer}  the
  $S^1$-equivariant cohomology of $\mathcal{S}_N$ is a free
  $H^*_{S^1}(\pt;\C)$-module.
Let $M = H^*_{S^1}(\mathcal{S}_N;\C)$ denote the $S^1$-equivariant
cohomology of the Springer variety $\mathcal{S}_N$ considered as an
$H^*_{S^1}(\pt;\C)$-module and let $t$ be
the degree $2$ polynomial variable in $H^*_{S^1}(\pt;\C) \cong \C[t]$.  The ordinary cohomology of
$\mathcal{S}_N$ is isomorphic to the quotient $M/(t)M$ \cite[Equation
1.2.4]{GKM}.
The
$H^*_{S^1}(\pt;\C)$-module structure on the quotient factors through
$\C$ via the ring homomorphism $H^*_{S^1}(\pt;\C) \cong \C[t] \to \C$
taking $t$ to $0$. In particular the images of the $H^*_{S^1}(\pt;\C)$-module
basis $\{p_e, p_{s_1}, \ldots, p_{s_{n-1}}\}$ in
$H^*_{S^1}(\mathcal{S}_N;\C)$ form a $\C$-module basis for
$H^*(\mathcal{S}_N;\C)$. The $S_n$-action defined on
the free $H^*_{S^1}(\pt;\C)$-module $H^*_{S^1}(\mathcal{S}_N;\C)$ is
$H^*_{S^1}(\pt;\C)$-linear and thus, via the quotient map taking $t$ to $0$, yields a well-defined action on
the free $\C$-module $H^*(\mathcal{S}_N;\C)$ as desired. 
\end{proof}

We refer to the $S_n$-actions on both
$H^*_{S^1}(\mathcal{S}_N;\C)$ and $H^*(\mathcal{S}_N;\C)$ 
as \textbf{Kostant-Kumar representations}. 
We now compute the character of the Kostant-Kumar
representation on the complex vector spaces $H^{2i}(\mathcal{S}_N;\C)$, denoted by $\psi^i$.  We then compare the
Kostant-Kumar representation with the Springer representation.

\begin{proposition}\label{prop:trace of Springer} 
Let $N$ be a subregular nilpotent linear operator $N: \C^n \to \C^n$ with Springer variety
$\mathcal{S}_N$. 
Let \(\psi^i: W \to \Z \) denote the character of the Kostant-Kumar  representation on $H^{2i}(\mathcal{S}_N;\C)$. Then for each $w\in S_n$
\[ 
\psi^1(w) = \# \{\textup{ fixed points of } w\} -1 = \#\{j \in \{1,2,\ldots,n\}:
w(j)=j \} - 1.
\]
The $S_n$-representation on $H^0(\mathcal{S}_N;\C)$ is the trivial
representation, hence for each $w \in S_n$
\[
\psi^0(w)  = 1.
\]
\end{proposition}

\begin{proof}
For the purposes of this proof we
use cycle notation for permutations, so
e.g. $(1,2,3,4)$ sends $1$ to $2$, $2$ to
$3$, $3$ to $4$, and $4$ to $1$. Each element of $S_n$ may be written as
a product of disjoint cycles, where the product is denoted by concatenation.
The character is a class function, so it suffices to compute
$\psi^1(w)$ on a representative of each
conjugacy class. Thus we may assume without loss of generality that $w$ has the form 
\begin{equation}\label{eq:w decomp into cycles}
w = (1, 2, \ldots, \mu_1)(\mu_1+1, \mu_1+2, \ldots, \mu_2) \cdots (\mu_{j-1}+1, \mu_{j-1}+2,
\ldots, \mu_j=n),
\end{equation}
for some $\mu_1 < \mu_2 < \cdots < \mu_j$ where cycles may
have length $1$.

Choose $a,b$ with $1 \leq a < b \leq n$.  A cycle $(a, a+1, \ldots, b-1, b)$  of length
at least $2$ has reduced word decomposition \(s_a
s_{a+1} \cdots s_{b-1}.\) Using this word and
the formula in Corollary~\ref{corollary: KK on Springers} we easily check that
\begin{equation}\label{eq:action under cycle}
(a, a+1, \ldots, b-1, b) \cdot p_k = 
\begin{cases} 
-p_k + \left( \textup{ $\Z$-linear combination of  } \{p_j\}_{j \neq k} \right), \quad \quad \textup{if }
  k=a, \\
\textup{ $\Z$-linear combination of } \{p_j\}_{j \neq k}, \quad \quad \quad \quad \hspace{2em}
\textup{if } a+1 \leq k \leq b-1, \\
p_k + \left( \textup{ $\Z$-linear combination of } \{p_j\}_{j \neq k}
\right), \quad \quad \hspace{0.75em} \textup{else.}
\end{cases}
\end{equation}
For any $a = 1,2,\ldots,n$, a cycle $(a)$ of length $1$ corresponds to
a fixed point of $w$. The cycle $(a)$ also denotes the identity element in
$S_n$ so 
\begin{equation}\label{eq: action under singleton}
(a) \cdot p_k = p_k \textup{ for all } k=1, 2, \ldots, n-1.
\end{equation}
For each $k=1,2,\ldots,n-1$ consider the basis element $p_k$. The index $k$
 appears in precisely one of the cycles in 
Equation~\ref{eq:w decomp into cycles}. From Equations~\eqref{eq:action under
  cycle} and~\eqref{eq: action under singleton} we conclude that, as desired,
\[
\psi^1(w) = \textup{ (number of cycles of length $1$) }
-1  = \# \{ j \in \{1,2,\ldots, n\}: w(j) = j \}  - 1.
\]
Finally, the class $p_e$ generates $H^0(\mathcal{S}_N;\C)$ and $w p_e =
p_e$ for all $w \in S_n$.  So $H^0(\mathcal{S}_N;\C)$ is the trivial
$1$-dimensional representation, and \(\psi^0(w) = 1\)
for all \(w \in S_n.\) 
\end{proof}

Finally, we observe that the Kostant-Kumar  $S_n$-representation on
$H^*(\mathcal{S}_N;\C)$ agrees with the Garsia-Procesi description of
the Springer representation \cite{GarPro92}.  In fact, since
$S_n$-representations are uniquely determined by their characters, the
following is immediate from Corollary~\ref{corollary:trace} and Proposition~\ref{prop:trace of Springer}.

\begin{corollary}\label{corollary:equivariant Springer rep}
Let $N$ be a subregular nilpotent linear operator $N: \C^n \to \C^n$ and let
$\mathcal{S}_N$ denote its associated subregular Springer variety. 
The
  $S_n$-representation on $H^*_{S^1}(\mathcal{S}_N;\C)$ constructed in
  Corollary~\ref{corollary: KK on Springers} lifts the
  classical subregular Springer representation on $H^*(\mathcal{S}_N;\C)$
  to $H^*_{S^1}(\mathcal{S}_N;\C)$ via the homomorphism $H^*_{S^1}(\mathcal{S}_N;\C)
  \to H^*(\mathcal{S}_N;\C)$. 
\end{corollary}

\begin{remark} 
The extent to which the constructions in 
Section~\ref{subsec:subregular}  can be generalized
to other classes of Springer varieties is an open question.  
We have preliminary experimental evidence that suggests that 
module bases constructed via poset pinball for other  Springer
varieties are not in general
poset-upper-triangular with respect to Bruhat order. This may make
computations using them more difficult. We leave further
exploration to future work.  
\end{remark}

We conclude the paper with a concrete example of Betti pinball in the
Springer case which does yield a module basis, albeit not a
poset-upper-triangular one. 

\begin{example}\label{example:SprN=4,X=211}
In our last example, we play Betti pinball with
target Betti numbers $\mathbf{b} = (1, 3, 5, 3)$. 
For visual simplicity, not all edges in the poset are drawn above rank $2$.
The reader can verify that the rolldown set
  $\mathcal{R}(\mathcal{I}, \mathcal{J})$ in this case is a union of principal order
  ideals. 
The double lines indicate covering relations in the partial
order which cause the failure of poset-upper-triangularity of the
rolldown set. Specifically,
poset-upper-triangularity fails in this example because of the drop-downs
  labeled \(v_8, v_{10}, v_{11}\), and $v_{12}$ in the table.  
  
  This example is a pinball game arising from the
  Springer variety associated to the nilpotent matrix for the
  partition $(2,1,1)$ of $n=4$.  In this case we can check by direct computation
  that the rolldown set $\mathcal{R}(\mathcal{I}, \mathcal{J})$ yields
  a module basis for the submodule corresponding to $\mathcal{J}$,
even though it is not poset-upper-triangular.

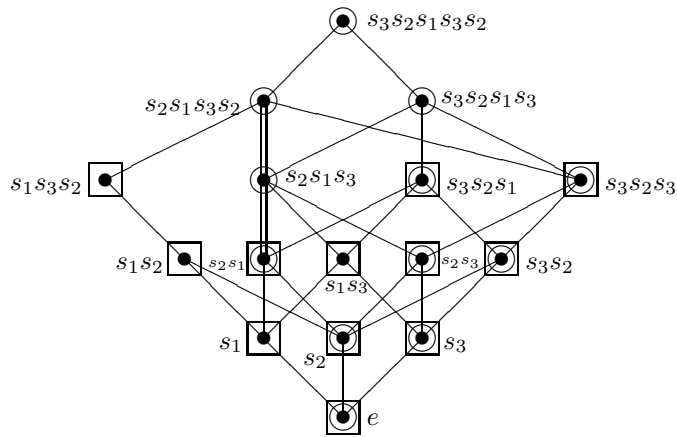
\begin{figure}[h]
\begin{picture}(230,165)(-40,0)
\put(75,0){\circle*{5}}
\multiput(45,30)(30,0){3}{\circle*{5}}
\multiput(15,60)(30,0){5}{\circle*{5}}
\multiput(-15,90)(60,0){2}{\circle*{5}}
\multiput(105,90)(60,0){2}{\circle*{5}}
\multiput(45,120)(60,0){2}{\circle*{5}}
\put(75,150){\circle*{5}}

\put(75,0){\line(-1,1){30}}
\put(75,0){\line(0,1){30}}
\put(75,0){\line(1,1){30}}

\put(45,30){\line(-1,1){30}}
\put(45,30){\line(0,1){30}}
\put(45,30){\line(1,1){30}}
\put(75,30){\line(-1,1){30}}
\put(75,30){\line(-2,1){60}}
\put(75,30){\line(1,1){30}}
\put(75,30){\line(2,1){60}}
\put(105,30){\line(-1,1){30}}
\put(105,30){\line(0,1){30}}
\put(105,30){\line(1,1){30}}

\put(-15,90){\line(1,-1){30}}
\put(165,90){\line(-2,-1){60}}
\put(165,90){\line(-1,-1){30}}
\put(44,90){\line(0,-1){30}}
\put(46,90){\line(0,-1){30}}
\put(45,90){\line(1,-1){30}}
\put(45,90){\line(2,-1){60}}
\put(105,90){\line(1,-1){30}}
\put(105,90){\line(-1,-1){30}}
\put(105,90){\line(-2,-1){60}}

\put(45,120){\line(-2,-1){60}}
\put(44,120){\line(0,-1){30}}
\put(46,120){\line(0,-1){30}}
\put(45,120){\line(4,-1){120}}
\put(105,120){\line(-2,-1){60}}
\put(105,120){\line(0,-1){30}}
\put(105,120){\line(2,-1){60}}
\put(75,150){\line(-1,-1){30}}
\put(75,150){\line(1,-1){30}}

\put(75,0){\circle{10}}
\put(75,30){\circle{10}}
\put(105,30){\circle{10}}
\put(45,60){\circle{10}}
\put(105,60){\circle{10}}
\put(135,60){\circle{10}}
\put(105,90){\circle{10}}
\put(165,90){\circle{10}}
\put(45,90){\circle{10}}
\put(45,120){\circle{10}}
\put(105,120){\circle{10}}
\put(75,150){\circle{10}}

\put(69,-6){\framebox(12,12)}
\multiput(39,24)(30,0){3}{\framebox(12,12)}
\multiput(9,54)(30,0){5}{\framebox(12,12)}
\put(-21,84){\framebox(12,12)}
\put(159,84){\framebox(12,12)}
\put(99,84){\framebox(12,12)}

\put(84,-3){$e$}
\put(113,26){$s_3$}
\put(60,20){$s_2$}
\put(28,26){$s_1$}
\put(-11,56){$s_1s_2$}
\put(144,56){$s_3s_2$}
\put(24,57){\tiny $s_2s_1$}
\put(112,58){\tiny $s_2s_3$}
\put(68,48){\small $s_1s_3$}
\put(-51,86){$s_1s_3s_2$}
\put(174,86){$s_3s_2s_3$}
\put(53,89){$s_2s_1s_3$}
\put(114,86){$s_3s_2s_1$}
\put(0,116){$s_2s_1s_3s_2$}
\put(112,119){$s_3s_2s_1s_3$}
\put(84,148){$s_3s_2s_1s_3s_2$}

\end{picture}
\caption{An example of Betti pinball for which the rolldown set is not
  poset-upper-triangular, but is a 
  union of principal order ideals.}\label{fig:SpringerN=4,YoungX=2,1,1}
\end{figure}

\renewcommand{\arraystretch}{1.3}
\begin{equation}\label{eq:SpringerN=4,YoungX=2,1,1-pinball}
\begin{array}{c||c|c|}
\mbox{pinball step} & w_k & v_k  \\ \hline \hline
1 & w_1 = e = [1, 2, 3, 4] &  v_1 = e = [1, 2, 3, 4] \\ \hline 
2 & w_2 = s_3 = [1, 2, 4, 3] & v_2 = s_3 = [1, 2, 4, 3]  \\ \hline 
3 & w_3 = s_2 = [1, 3, 2, 4]  & v_3 = s_2 = [1, 3, 2, 4]  \\ \hline 
4& w_4 = s_2 s_3 = [1, 3, 4, 2] &  v_4 = s_2 s_3 = [1, 3, 4, 2]  \\ \hline 
5 & w_5 = s_3 s_2 = [1, 4, 2, 3] &  v_5 = s_3 s_2 =  [1, 4, 2, 3] \\ \hline 
6 & w_6 = s_2 s_1 = [3, 1, 2, 4] & v_6 = s_1 = [2, 1, 3, 4] \\ \hline 
7 & w_7 = s_3 s_2 s_3 = [1, 4, 3, 2] &  v_7 = s_3 s_2 s_3 = [1, 4, 3, 2]\\ \hline 
8 & w_8 = s_2 s_1 s_3 = [3, 1, 4, 2] &  v_8 = s_1 s_3 = s_3 s_1 = [2, 1, 4, 3] \\ \hline 
9 & w_{9} = s_3 s_2 s_1 = [4, 1, 2, 3] & v_{9} = s_2 s_1 = [3, 1, 2, 4] \\ \hline 
10 & w_{10} = s_2 s_1 s_3 s_2 = [3, 4, 1, 2] &  v_{10} = s_1 s_2 = [2, 3, 1, 4] \\ \hline 
11 & w_{11} = s_3 s_2 s_1 s_3 = [4, 1, 3, 2] &  v_{11} = s_3 s_2 s_1 =
[4, 1, 2, 3] \\ \hline 
12 & w_{12} = s_3 s_2 s_1 s_3 s_2 = [4, 3, 1, 2] &  v_{12} = s_1 s_3
s_2 = [2, 4, 1, 3] \\ \hline 
\end{array}
\end{equation}

\end{example}

\def\cprime{$'$}

\end{document}